\documentclass{amsart}

\usepackage{amsmath}
\usepackage{amsthm}
\usepackage{amsfonts}
\usepackage{amssymb}
\usepackage{graphics}

\newcommand\R{{\mathbb{R}}}
\newcommand\C{{\mathbb{C}}}
\newcommand\Z{{\mathbf{Z}}}

\newcommand\Q{{\mathbf{Q}}}

\newcommand\N{{\mathbf{N}}}
\newcommand\n{{\mathbf{n}}}
\renewcommand\P{{\mathbf{P}}}
\newcommand\E{{\mathbf{E}}}

\newcommand\eps{{\varepsilon}}

\newcommand\SL{\operatorname{SL}}

\newcommand\st{\operatorname{st}}

\theoremstyle{plain}
  \newtheorem{theorem}{Theorem}

  \newtheorem{proposition}[theorem]{Proposition}
  
  \newtheorem{lemma}[theorem]{Lemma}
  \newtheorem{corollary}[theorem]{Corollary}

\theoremstyle{definition}
  \newtheorem{definition}[theorem]{Definition}
  \newtheorem{example}[theorem]{Example}
  \newtheorem{remark}[theorem]{Remark}

\begin{document}

\title{Multiple recurrence in quasirandom groups}

\author{Vitaly Bergelson}
\address{Department of Mathematics, Ohio State University, Columbus OH 43210-1174}
\email{vitaly@math.ohio-state.edu}

\author{Terence Tao}
\address{UCLA Department of Mathematics, Los Angeles, CA 90095-1555.}
\email{tao@math.ucla.edu}


\begin{abstract}  We establish a new mixing theorem for quasirandom groups (finite groups with no low-dimensional unitary representations) $G$ which, informally speaking, asserts that if $g, x$ are drawn uniformly at random from $G$, then the quadruple $(g,x,gx,xg)$ behaves like a random tuple in $G^4$, subject to the obvious constraint that $gx$ and $xg$ are conjugate to each other.   The proof is non-elementary, proceeding by first using an ultraproduct construction to replace the finitary claim on quasirandom groups with an infinitary analogue concerning a limiting group object that we call an \emph{ultra quasirandom group}, and then using the machinery of idempotent ultrafilters to establish the required mixing property for such groups.   Some simpler recurrence theorems (involving tuples such as $(x,gx,xg)$) are also presented, as well as some further discussion of specific examples of ultra quasirandom groups.
\end{abstract}

\maketitle

\section{Introduction}

In \cite{gowers}, Gowers introduced the notion of a \emph{quasirandom group}:

\begin{definition}[Quasirandom group]
A finite group $G$ is said to be \emph{$D$-quasirandom} for some parameter $D \geq 1$ if all non-trivial unitary\footnote{If desired, one could replace ``unitary'' here by ``$\C$-linear'' (thus relaxing $U_d(\C)$ to $GL_d(\C)$), because any linear action of a finite group preserves at least one Hermitian form by an averaging argument.} representations $\rho: G \to U_d(\C)$ of $G$ have dimension $d$ greater than or equal to $D$.  
\end{definition}

We informally refer to a \emph{quasirandom group} to be a finite group that is $D$-quasirandom for a large value of $D$.

\begin{example}  The alternating group $A_n$ is $n-1$-quasirandom for all $n \geq 6$.  More generally, if $G$ is perfect (i.e. $G=[G,G]$) and has no normal subgroup of index less than $m$, then $G$ is $\sqrt{\log m}/2$-quasirandom; see \cite[Theorem 4.8]{gowers}, which also asserts a converse implication (but with $\sqrt{\log m}/2$ replaced by $\sqrt{m}$).  In particular, if $G$ is a non-abelian finite simple group, then $G$ is $\sqrt{\log|G|}/2$-quasirandom, where $|G|$ denotes the cardinality of $G$; more generally, if every group in the Jordan-Holder decomposition of $G$ is a non-abelian finite simple group of order at least $m$, then $G$ is $\sqrt{\log m}/2$-quasirandom.  On the other hand, for many finite simple groups, one can improve this logarithmic bound to a polynomial bound.  For instance, the group $SL_2(F_p)$ is $\frac{p-1}{2}$-quasirandom for any prime $p$; see Lemma \ref{frob}.  
More generally, see \cite{landazuri-seitz} for a precise computation of the quasirandomness for finite Chevalley groups.  

One can combine these observations to obtain further examples of quasirandom groups.  For instance, if $p,q$ are distinct primes, then by the Chinese remainder theorem, $SL_2(\Z/pq\Z)$ is isomorphic to the direct product of $SL_2(F_p)$ and $SL_2(F_q)$, and so is $\frac{\min(p,q)-1}{2}$-quasirandom.
\end{example}

When $D$ is large, such groups become mixing in the sense that averages such as
$$ \E_{x \in G} f_1(x) f_2(xg)$$
for ``typical'' values of $g$ tend to stay very close to $(\E_G f_1) (\E_G f_2)$ for bounded functions $f_1,f_2: G \to \R$, where we use the averaging notation
$$\E_G f = \E_{x \in G} f(x) := \frac{1}{|G|} \sum_{x \in G} f(x).$$
More precisely, we have the following inequality, essentially present in \cite{gowers} (see also \cite{babai-nikolov-pyber}):

\begin{proposition}[Weak mixing]\label{wm}  Let $G$ be a $D$-quasirandom group for some $D \geq 1$, and let $f_1,f_2: G \to \C$ be functions.  Then
$$ \E_{g \in G} |\E_{x \in G} f_1(x) f_2(xg) - (\E_G f_1) (\E_G f_2)| \leq D^{-1/2} \|f_1\|_{L^2(G)} \|f_2\|_{L^2(G)}$$
where $\|f\|_{L^2(G)} := (\E_G |f|^2)^{1/2}$.
\end{proposition}

\begin{proof} Observe that the left-hand side does not change if one subtracts a constant from either $f_1$ or $f_2$, so we may reduce to the case when $f_2$ has mean zero.  By the Cauchy-Schwarz inequality, it thus suffices to show that the linear operator $T = T_{f_2}: L^2(G) \to L^2(G)$ defined by
$$ T_{f_2} f_1(g) := \E_{x \in G} f_1(x) f_2(xg)$$
has operator norm $\|T\|_{\operatorname{op}}$ at most $D^{-1/2} \|f_2\|_{L^2(G)}$.

We may of course assume that $f_2$ is not identically zero.  Let $V$ be the space of functions $f \in L^2(G)$ such that $\| T_{f_2} f \|_{L^2(G)} = \|T\|_{\operatorname{op}} \|f\|_{L^2(G)}$, i.e. the right singular space corresponding to the largest singular value $\|T\|_{\operatorname{op}}$.  This is a vector space which is invariant under the action of right-translation $R_h f(x) := f(xh)$ by elements of $h \in G$, and also does not contain any non-trivial constant functions, and hence by quasirandomness has dimension at least $D$.  As $T$ acts by a multiple of an isometry by $\|T\|_{\operatorname{op}}$ on $V$, we conclude that the Hilbert-Schmidt norm 
$$ \|T\|_{\operatorname{HS}} = (\E_{x,g \in G} |f_2(xg)|^2)^{1/2} = \|f_2\|_{L^2(G)}$$
is at least $D^{1/2} \|T\|_{\operatorname{op}}$, and the desired bound $\|T\|_{\operatorname{op}} \leq D^{-1/2} \|f_2\|_{L^2(G)}$ follows.
\end{proof}

In particular, if $G$ is $D$-quasirandom and $f_1,f_2: G \to \R$ are bounded in magnitude by $1$, we have
$$ \E_{g \in G} |\E_{x \in G} f_1(x) f_2(xg) - (\E_G f_1) (\E_G f_2)| \leq D^{-1/2} $$
and hence by Markov's inequality $\P(|X| \geq \lambda) \leq \frac{1}{\lambda} \E |X|$ we have
$$ |\E_{x \in G} f_1(x) f_2(xg) - (\E_G f_1) (\E_G f_2)| \leq D^{-1/4} $$
for at least $1-D^{-1/4}$ of the elements $g \in G$.  Thus, when $D$ is large, we heuristically have the mixing property $\E_{x \in G} f_1(x) f_2(xg) \approx (\E_G f_1) (\E_G f_2)$ for most group elements $g \in G$.  Specialising to the case of indicator functions $f_1=1_A, f_2=1_B$ of sets, we conclude the heuristic that $\frac{|A \cap Bg|}{|G|} \approx\frac{|A|}{|G|} \frac{|B|}{|G|}$ for most $g \in G$, where $Bg := \{ bg: b \in B \}$ is the right-translate of $B$ by $g$.

\begin{remark}  Proposition \ref{wm} has a counterpart in ergodic theory.  Call an infinite group $G$ \emph{weakly mixing} if it has no non-trivial finite-dimensional representations.  If $G$ is a countable amenable group, one can show that $G$ is weakly mixing if and only if, for every ergodic action $(T_g)_{g \in G}$ of $G$ on a probability space $(X,\mu)$, one has 
$$ \lim_{n \to \infty} \frac{1}{|F_n|} \sum_{g \in F_n} |\int_X f_1 T_g f_2\ d\mu - (\int_X f_1\ d\mu) (\int_X f_2\ d\mu)|  = 0$$
for all $f_1,f_2 \in L^2(X,\mu)$, where $T_g f(x) := f(xg)$ is the right-translate of $f$ by $g$, and $F_1,F_2,\ldots$ is a F{\o}lner sequence in $G$.  See \cite{bf}, \cite{bg}, \cite{br} for further discussion of weak mixing for groups (including the case of non-amenable groups).
\end{remark}

In \cite{gowers}, it was observed that Proposition \ref{wm} could be iterated to obtain similar weak mixing results for some higher order averages; for instance, three applications of the above proposition and the triangle inequality show that
$$ \E_{g,h \in G} |\E_{x \in G} f_1(x) f_2(xg) f_3(xh) f_4(xgh) - (\E_G f_1) (\E_G f_2) (\E_G f_3) (\E_G f_4)| \leq 3 D^{-1/2} $$
whenever $f_1,f_2,f_3,f_4: G \to \R$ have magnitude bounded by $1$.  However, not all multiple averages could be controlled non-trivially in this fashion; for instance, in \cite[\S 6]{gowers}, the task of obtaining a mixing bound for the average
$$ \E_{g \in G} |\E_{x \in G} f_1(x) f_2(xg) f_3(xg^2) - (\E_G f_1) (\E_G f_2) (\E_G f_3)| $$
was posed as an open problem.  Some results for this average will be presented in the forthcoming paper \cite{tao} of the second author, using techniques quite different from those used here.  In this paper, we will focus instead on the weak mixing properties of the average
$$ \E_{x \in G} f_1(x) f_2(xg) f_3(gx)$$
for functions $f_1,f_2,f_3: G \to \R$, which in particular would control the behavior of the density $\frac{|A \cap Bg \cap gC|}{|G|}$ for sets $A,B,C \subset G$ and typical $g \in G$.

As already observed in \cite[\S 6]{gowers}, one cannot hope for absolute weak mixing for this average, due to the simple constraint that $xg$ is conjugate to $gx$.  For instance, if $B$ is the union of some conjugacy classes in $G$ and $C$ is the union of a disjoint collection of conjugacy classes, then $A \cap Bg \cap gC$ is empty for every $g$.  However, our main result asserts, roughly speaking, that these conjugacy classes form the \emph{only} obstruction to weak mixing:

\begin{theorem}[Relative weak mixing]\label{main}  Let $G$ be a $D$-quasirandom finite group for some $D \geq 1$, and let $f_1,f_2,f_3: G \to \R$ be functions bounded in magnitude by $1$.  Then
$$ \E_{g \in G} |\E_{x \in G} f_1(x) f_2(xg) f_3(gx) - (\E_G f_1) (\E_G f_2 \E(f_3|{\mathcal I}_G) ))| \leq c(D)$$
where $\E(f|{\mathcal I}_G)$ is the orthogonal projection of a function $f$ to conjugation invariant functions, thus
$$\E(f|{\mathcal I}_G)(x) := \E_{g \in G} f(gxg^{-1}),$$
and $c(D)$ is a quantity depending only on $D$ that goes to zero as $D \to \infty$.
\end{theorem}

Thus, for instance, we have the following general bounds on $|A \cap gB \cap Bg|$ for most $g$:

\begin{corollary}\label{main-cor} Let $G$ be a $D$-quasirandom group for some $D \geq 1$, let $\eps > 0$, and let $A,B \subset G$.  Then we have
$$ \frac{|A|}{|G|} \left(\frac{|B|}{|G|}\right)^2 - \eps \leq \frac{|A \cap gB \cap Bg|}{|G|} \leq \frac{|A|}{|G|} \frac{|B|}{|G|} + \eps$$
for all but at most $\eps^{-1} c(D) |G|$ values of $g \in G$, where $c(D)$ goes to zero as $D \to \infty$.  
\end{corollary}

\begin{proof}  The upper bound follows from Proposition \ref{wm} with $f_1 := 1_A$ and $f_2 := 1_B$ (replacing $g$ by $g^{-1}$), after crudely bounding $|A \cap gB \cap Bg|$ by $|A \cap Bg|$.  Now we turn to the lower bound.  By setting $f_1 := 1_A$ and $f_2 = f_3 := 1_B$ (and replacing $g$ by $g^{-1}$), this bound is immediate from Theorem \ref{main} and Markov's inequality once one verifies that 
\begin{equation}\label{egb}
 \E_G 1_B \E(1_B|{\mathcal I}_G) \geq \left(\frac{|B|}{|G|}\right)^2.
\end{equation}
But this follows from the identities
\begin{align*}
 \E_G 1_B \E(1_B|{\mathcal I}_G) &= \E_G \E(1_B|{\mathcal I}_G) \E(1_B|{\mathcal I}_G)\\
 \E_G \E(1_B|{\mathcal I}_G) &= \frac{|B|}{|G|}
\end{align*}
and the Cauchy-Schwarz inequality.
\end{proof}

Note from Theorem \ref{main} that apart from improvements in the $\eps^{-1} c(D)$ factor, both bounds in the above corollary are best possible without further hypotheses on $B$, with the lower bound essentially attained when $B$ is approximately evenly distributed among all (or almost all) conjugacy classes, and the upper bound essentially attained when $B$ is the union of conjugacy classes.

One can also view Theorem \ref{main} more probabilistically:

\begin{corollary}[Relative weak mixing, again]\label{main-again}  Let $G$ be a $D$-quasirandom finite group for some $D \geq 1$.  Let $x,g$ be drawn uniformly at random from $G$.  Let $x_0,x_1,x_2,x_3$ be further random variables drawn from $G$, with $x_0,x_1,x_2$ drawn uniformly and independently at random from $G$, and for each choice of $x_0,x_1,x_2$, the random variable $x_3$ is drawn uniformly from the conjugacy class of $x_2$.  (Equivalently, one could take $x_3 := hx_2 h^{-1}$, where $h$ is drawn uniformly from $G$ independently of $x_0,x_1,x_2$).  Then the random variables $(g,x,xg,gx)$ and $(x_0,x_1,x_2,x_3)$ in $G^4$ are close in the following weak sense: whenever $f_0,f_1,f_2,f_3: G \to \R$ are functions bounded in magnitude by $1$, then
$$ |\E f_0(g) f_1(x) f_2(xg) f_3(gx) - \E f_0(x_0) f_1(x_1) f_2(x_2) f_3(x_3)| \leq c(D)$$
where $c(D) \to 0$ as $D \to \infty$.
\end{corollary}

The equivalence of Corollary \ref{main-again} and Theorem \ref{main} follows from a routine computation which we leave to the reader.  Informally, this corollary asserts that in a quasirandom group, the only significant constraint on the tuple $(x,g,xg,gx)$ for random $x,g$ is the obvious constraint that $xg$ and $gx$ are conjugate to each other, at least for the purposes of computing ``order $1$ statistics'' such as $\E f_0(x) f_1(x) f_2(xg) f_3(gx)$ involving products of quantities, each of which only involve at most one of the expressions $x,g,xg,gx$.

We can also easily obtain the following combinatorial consequence of the above results, which can be viewed as a density version of the non-commutative Schur theorem from \cite{bc2}, in the case of quasirandom groups:

\begin{corollary}[Density noncommutative Schur theorem]\label{dst} Let $G$ be a $D$-quasirandom finite group for some $D \geq 1$.  Let $A, B, C \subset G$ with $|A|, |B|, |C| \geq \delta |G|$ for some $\delta>0$.  Then, if $D$ is sufficiently large depending on $\delta$, there exists $g \in A, x \in B$ with $xg, gx \in C$ and $x,g,xg,gx$ all distinct.
\end{corollary}

\begin{proof}  Let $x,g$ be drawn uniformly at random from $G$, and let $x_0,x_1,x_2,x_3$ be drawn as in Corollary \ref{main-again}.  We have
$$ \E 1_A(x_0) 1_B(x_1) 1_C(x_2) 1_C(x_3) = (|A|/|G|) (|B|/|G|)  \E_G (1_C \E(1_C|{\mathcal I}_G))$$
and hence by \eqref{egb}
$$ \E 1_A(x_0) 1_B(x_1) 1_C(x_2) 1_C(x_3) \geq \delta^4.$$
By Corollary \ref{main-again}, we thus have (for $D$ large enough) that
$$ \E 1_A(g) 1_B(x) 1_C(xg) 1_C(gx) \geq \delta^4/2,$$
thus there are at least $\delta^4 |G|^2/2$ tuples $(g,x,xg,gx)$ with $g \in A$, $x \in B$, and $xg, gx \in C$.

Now we eliminate those tuples in which $g,x,xg,gx$ are not all distinct.  Clearly there are at most $O(|G|)$ tuples for which $g=x$ or for which one of $g,x$ is equal to one of $xg,gx$.  Now we consider those tuples for which $xg=gx$.  By Burnside's lemma, the number of such tuples is equal\footnote{This observation dates back to \cite{erdos}.} to $|G|$ times the number of conjugacy classes of $G$.  But from the Peter-Weyl theorem, the number of conjugacy classes is also equal to the number of non-isomorphic irreducible unitary representations $\rho: G \to U_{d_\rho}(\C)$ of $G$, which obey the identity 
$$ \sum_\rho d_\rho^2 = |G|.$$
If $G$ is $D$-quasirandom, then $d_\rho \geq D$ for all non-trivial $\rho$, and so we see that the number of pairs $x,g$ with $xg=gx$ is at most
$$ |G| (1 + \frac{|G|-1}{D^2}) \leq \frac{|G|^2}{D^2} + 1.$$
Putting everything together, we see that there are at most $O(|G|^2/|D|^2) + O(|G|)$ tuples for which $x,g,xg,gx$ are not all distinct.  Since $|G| \geq D$ (as can be seen by considering the regular representation of $G$), the claim then follows if $D$ is sufficiently large depending on $\delta$.
\end{proof}

In order to prove Theorem \ref{main}, we will introduce a new version of the Furstenberg correspondence principle adapted to sequences of finite quasirandom groups, which is based on the ultraproduct construction used in nonstandard analysis.  As with the usual correspondence principle, this construction will allow us to deduce the combinatorial mixing result in Theorem \ref{main} from a corresponding mixing result in ergodic theory.  A key feature of the construction is that significant vestiges of the mixing property from Proposition \ref{wm} are retained in the measure-preserving system that one studies on the ergodic theoretic side of the correspondence principle.  On the other hand, the group that acts in this setting is not expected to be amenable, instead behaving more like the free nonabelian group.  Fortunately, there is an existing tool in the literature for exploiting mixing properties for non-amenable groups, namely the machinery of idempotent ultrafilters. We will  introduce the necessary definitions in later sections. For more information on ultrafilters and their use in ergodic theory see \cite{berg1, berg2, berg3, berg4}.

Due to the repeated use of infinitary techniques (including two completely separate and unrelated uses of ultrafilters), our arguments do not give an effective\footnote{The most infinitary step in the arguments involve the usage of idempotent ultrafilters, which are closely related to Hindman's theorem \cite{hind} in infinitary Ramsey theory.  It may be possible to use some suitable finitizations of Hindman's theorem as a substitute for the tool of idempotent ultrafilters to eventually obtain some (incredibly poor) quantitative decay rate for $c(D)$, but we do not pursue this issue here.}
 bound on the rate of decay of $c(D)$ as $D \to \infty$.  In particular, we do not know if one can obtain a polynomial rate of decay in $D$, in analogy with Proposition \ref{wm}.  

In order to illustrate the general ultraproduct correspondence principle strategy, we also give a significantly simpler and weaker recurrence result which does not assume quasirandomness, but only establishes recurrence for the pattern $(x,xg,gx)$ rather than $(g,x,xg,gx)$:

\begin{theorem}[Double recurrence]\label{loaded}  For every $\delta>0$, there exists $N > 0$ and $\eps>0$ such that the following statement holds: if $G$ is a finite group of cardinality at least $N$, and $A$ is a subset of $G$ with $|A| \geq \delta |G|$, then there exists a non-trivial group element $g \in G$ such that $|A \cap g^{-1} A \cap A g^{-1}| \geq \eps |G|$.  In particular, there exists a non-trivial $g \in G$ and $x \in G$ such that $x, xg, gx \in A$.
\end{theorem}

We prove this theorem in Section \ref{lo-sec}.   It may be compared with \cite[Corollary 6.5]{berg-amenable}, which with the same hypotheses produces a non-trivial $g \in G$ and $x \in G$ such that $x, gx, xg^{-1} \in A$.

Actually, we can improve Theorem \ref{loaded} slightly:

\begin{theorem}[Double strong recurrence]\label{loaded-2}  For every $\delta>0$, there exists $\eps>0$ such that the following statement holds: if $G$ is a finite group, and $A$ is a subset of $G$ with $|A| \geq \delta |G|$, then there exist at least $\eps |G|^2$ pairs $(g,x) \in G \times G$ such that $x, xg, gx \in A$.
\end{theorem}

We give two proofs of this theorem in Section \ref{lo2-sec}.  One proof is measure-theoretic in nature (and similar in spirit to the ergodic theory methods).  The other proof is combinatorial, relying on the triangle removal lemma of Ruzsa and Szemer\'edi \cite{rsz}.  However, we do not know how to use such methods to establish Theorem \ref{main} or Corollary \ref{main-cor}.  As with the proof of Corollary \ref{dst}, we can ensure that $x,xg,gx$ are all distinct if one assumes a sufficient amount\footnote{\label{foot}The groups $G$ with few commuting pairs $\{ (g,x) \in G \times G: xg = gx \}$ were qualitatively classified in \cite{neumann}.  Roughly speaking, the necessary and sufficient condition that $G$ has $o(|G|^2)$ such pairs is that $G$ does not have a bounded index subgroup whose commutator also has bounded index.  Note this is a much weaker property than quasirandomness, for which the argument used to prove Corollary \ref{dst} may be applied. See also \cite{guralnick} for some more quantitatively precise characterizations of groups with many commuting pairs.} of non-commutativity in the group $G$.  Of course, if $G$ is abelian, then $xg=gx$ and Theorems \ref{loaded}, \ref{loaded-2} become trivial.  

As we shall show in Section \ref{lo2-sec}, the combinatorial proof of Theorem \ref{loaded-2} generalizes to give a multidimensional version:

\begin{theorem}[Multiple strong recurrence]\label{loaded-4}  Let $k \geq 1$ be an integer.  For every $\delta>0$, there exists $\eps>0$ such that the following statement holds: if $G$ is a finite group, and $A$ is a subset of $G^k$ with $|A| \geq \delta |G|^k$, then there exist at least $\eps |G|^{k+1}$ tuples $(g,x_1,\ldots,x_k) \in G^{k+1}$ such that
$$ (gx_1,\ldots,gx_{i},x_{i+1},\ldots,x_k) \in A$$
for all\footnote{We ignore the block $gx_1,\ldots,gx_i$ when $i=0$ and ignore the block $x_{i+1},\ldots,x_k$ when $i=k$; thus we interpret $(gx_1,\ldots,gx_{i},x_{i+1},\ldots,x_k)$ as $(x_1,\ldots,x_k)$ in the case $i=0$, and $(gx_1,\ldots,gx_k)$ in the case $i=k$.} $i=0,\ldots,k$, and also
$$ (gx_1g^{-1},\ldots,gx_kg^{-1}) \in A.$$
\end{theorem}

Note that the $k=1$ case of this theorem, when applied to the set $A$ in Theorem \ref{loaded-2}, gives at least $\eps |G|^2$ pairs $(g,x)$ such that $x, gx, gxg^{-1} \in A$; by replacing $(x,g)$ by $(g' x',(g')^{-1})$ we see that Theorem \ref{loaded-2} is equivalent to the $k=1$ case of Theorem \ref{loaded-4}.

The analogue of Theorem \ref{loaded} (and hence Theorem \ref{loaded-2}) can fail for the pattern $x,g, gx, xg$ if one does not assume quasirandomness.  For instance, if $G$ has an index two subgroup $H$, and $A$ is the complement of $H$ in $G$, then $A$ contains no patterns of the form $x,g,gx$, let alone $x,g,gx,xg$.  However, it is still reasonable to hope for a ``noncommutative Schur theorem'', namely that if a finite group $G$ is partitioned into $r$ color classes $A_1,\ldots,A_r$, then at least one of the color classes $A_i$ has the property that $x,g,gx,xg \in A_i$ for at least $c |G|^2$ pairs $(x,g) \in G^2$, where $c$ depends only on $r$; in particular, under a suitable hypothesis that $G$ is sufficiently noncommutative (in the sense of footnote \ref{foot}), we can find $x,g$ such that $x,g,gx,xg$ are \emph{distinct} elements of $A_i$. We could not verify or disprove this claim, but note that an analogous claim in the infinitary setting of countable amenable groups was established in \cite[Theorem 3.4]{berg4}.  If one replaces the pattern $x,g,gx,xg$ by $x,g,gx$ then the claim easily follows from Folkman's theorem \cite[\S 3.4]{grs} applied to a randomly chosen (finite portion of an) IP system in $G$; we omit the details.

\begin{remark} Throughout this paper we shall freely use the axiom of choice.  However, thanks to a well known result of G\"odel \cite{godel}, any result that can be formalized in first-order arithmetic (such as\footnote{Strictly speaking, the definition of quasirandomness involves the field $\C$ of complex numbers, but it is easy to see that one can replace that field by the algebraic closure $\overline{\Q}$ of the rationals, which are easily formalized within first-order arithmetic.} Theorem \ref{main}) and is provable in Zermelo-Frankel set theory with the axiom of choice (ZFC), can also be proven in Zermelo-Frankel set theory without the axiom of choice (ZF).
\end{remark}

\subsection{Acknowledgments}

The first author acknowledges the support of the NSF under grant DMS-1162073.  The second author was partially supported by a Simons Investigator award from the Simons Foundation and by NSF grant DMS-0649473.  The authors also thank Robert Guralnick for help with the references, and the anonymous referee for a careful reading of the paper and several useful suggestions.

\section{Ultraproducts, $\sigma$-topologies, and Loeb measure}

The arguments in this paper will rely heavily on the machinery of ultraproducts, as well as some related concepts such as $\sigma$-topological spaces and the Loeb measure construction.  The purpose of this section is to review this machinery.

Given a set $S$, define an \emph{ultrafilter} on $S$ to be a collection $\alpha$ of subsets of $S$ such that the map $A \mapsto 1_{A \in \alpha}$, that assigns $1$ to subsets $A$ of $S$ in $\alpha$, and $0$ to all other subsets, is a finitely additive $\{0,1\}$-valued probability measure on $S$.  The set of all ultrafilters is denoted $\beta S$.  One can embed $S$ in $\beta S$ by identifying each $x \in S$ with the ultrafilter $\{ A \subset S: x \in A \}$ (or, equivalently, with the Dirac measure at $x$).  A routine application of Zorn's lemma shows that there exist non-principal ultrafilters $\alpha \in\beta S \backslash S$ for any infinite set $S$.

Now we fix a non-principal ultrafilter $\alpha \in \beta \N \backslash \N$ on the natural numbers $\N$.  A subset of $\N$ is said to be \emph{$\alpha$-large} if it lies in $\alpha$.  
Given a sequence $X_\n$ of sets indexed by all $\n$ in an $\alpha$-large set, we define the \emph{ultraproduct} $\prod_{\n \to \alpha} X_\n$ to be the space of all formal\footnote{Note that despite the formal use of the $\lim$ notation, no topological structure is required on the $X_\n$ in order to define an ultraproduct.  If one prefers, one could view $\prod_{\n \to \alpha} X_\n$ as the space of equivalence classes of tuples $(x_\n)_{\n \in A}$ defined on $\alpha$-large sets $A$, with $(x_\n)_{\n \in A}, (y_\n)_{\n \in B}$ equivalent if one has $x_\n=y_\n$ for an $\alpha$-large set of $\n$.} limits (or \emph{ultralimits}) $\lim_{\n \to \alpha} x_\n$ of sequences $x_\n$ defined and in $X_\n$ for an $\alpha$-large set of $\n$, with two formal limits $\lim_{\n \to \alpha} x_\n, \lim_{\n \to \alpha} y_\n$ declared to be equal if we have $x_\n = y_\n$ for an $\alpha$-large set of $\n$.  An ultraproduct $\prod_{\n \to \alpha} X_\n$ of sets $X_\n$ will be referred to as an \emph{internal set}.
For a single set $X$, the ultraproduct $\prod_{\n \to \alpha} X$ is called the \emph{ultrapower} of $X$ and is denoted ${}^* X$; note that $X$ embeds into ${}^* X$ after identifying each $x \in X$ with its ultralimit $\lim_{\n \to \alpha} x$. 
Given a sequence $f_\n: X_\n \to Y_\n$ of functions defined for an $\alpha$-large set of $\n$, we define the \emph{ultralimit} $\lim_{\n \to \alpha} f_\n$ to be the function from $\prod_{\n \to \alpha} X_\n$ to $\prod_{\n \to \alpha} Y_\n$ defined by
$$ (\lim_{\n \to \alpha} f_\n) (\lim_{\n \to \alpha} x_\n) := \lim_{\n \to \alpha} f_\n(x_\n)$$
Such functions are also known as \emph{internal functions}.

Given an element $x = \lim_{\n \to \alpha} x_\n$ of the ultrapower ${}^* \R$ of the reals, we say that $x$ is \emph{bounded} if $|x| \leq C$ for some real number $C$ (i.e. $|x_\n| \leq C$ for an $\alpha$-large set of $\n$), and \emph{infinitesimal} if $|x| \leq \eps$ for every real $\eps>0$.  A well-known Bolzano-Weierstrass argument shows that every bounded $x \in {}^* \R$ can be expressed uniquely as the sum of a real number $\operatorname{st}(x)$, referred to as the \emph{standard part} of $x$, and an infinitesimal.  By convention, we define the standard part of an unbounded element of ${}^*\R$ to be $\infty$.  The quantity $\st \lim_{\n \to \alpha} x_\n$ is also known as the \emph{$\alpha$-limit} of the $x_\n$, and can be denoted as $\alpha\!-\!\lim_\n x_\n$.

Internal sets $X = \prod_{\n \to \alpha} X_\n$ do not have a natural topological structure (other than the discrete topology).  However, as pointed out in \cite{szegedy}, there is a useful substitute for this topological structure on such an internal set $X$, namely the more general concept of a \emph{$\sigma$-topological} structure.

\begin{definition}[$\sigma$-topology]\label{sigma}\cite{szegedy}  A \emph{$\sigma$-topology} on a set $X$ is a collection ${\mathcal F}$ of subsets of $X$ which contains the empty set $\emptyset$ and whole set $X$, is closed under finite intersections, and is closed under \emph{at most countable} unions (as opposed to \emph{arbitrary} unions, which is the definition of a true topology).  The pair $(X,{\mathcal F})$ will be called a \emph{$\sigma$-topological space}.  An element of ${\mathcal F}$ will be called \emph{countably open}, and the complement of a countably open set in $X$ will be called \emph{countably closed}.   A map $f: X \to Y$ between two $\sigma$-topological spaces $(X,{\mathcal F}_X)$, $(Y, {\mathcal F}_Y)$ will be called \emph{continuous} if the inverse image of any countably open subset of $Y$ is a countably open subset of $X$, or equivalently if the inverse image of a countably closed set is a countably closed set.  Similarly, the map $f$ is said to be \emph{open} (resp. closed) if the forward image of any countably open (resp. closed) subset of $X$ is a countably open (resp. closed) subset of $Y$.

A $\sigma$-topological space $(X,{\mathcal F})$ is said to be \emph{countably compact} if any countable cover $X \subset \bigcup_{m=1}^\infty V_m$ of $X$ by countably open sets has a finite subcover, or equivalently if any collection $(F_m)_{m \in \N}$ of countably closed subsets in $X$ with the property that any finite intersection of the $F_m$ is non-empty, also necessarily has non-empty joint intersection $\bigcap_{m=1}^\infty F_m$.  A $\sigma$-topological space $(X,{\mathcal F})$ is said to be $T_1$ if every point in $X$ is countably closed.
\end{definition}

One should view ``countably compact $T_1$'' as being to $\sigma$-topology as ``compact Hausdorff'' is to ordinary topology.

We have the following basic compactness theorem, known to model theorists as the \emph{countable saturation property}:

\begin{lemma}[Countable saturation]\label{countsat}  Let $X = \prod_{\n \to \alpha} X_\n$ be an internal set, and let ${\mathcal F}_X$ be the collection of all subsets of $X$ that can be expressed as the union of at most countably many internal subsets of $X$.  Then $(X,{\mathcal F}_X)$ is a countably compact $T_1$ $\sigma$-topological space.  

Furthermore, any internal function $f: X \to Y$ between two internal sets $X, Y$ will be continuous, open and closed with respect to these $\sigma$-topological structures.
\end{lemma}

\begin{proof}  It is easy to see that ${\mathcal F}_X$ is a $T_1$ $\sigma$-topology.  To verify countable compactness, we will use the formulation  from Definition \ref{sigma} involving countably closed sets.
Expressing each $F_m$ as the countable intersection of internal sets, we see that we may assume without loss of generality that the $F_m$ are internal, thus for each $m$ we have $F_m = \prod_{\n \to \alpha} F_{\n,m}$ for some subsets $F_{\n,m}$ of $X_\n$.   (Strictly speaking, $F_{\n,m}$ is initially only defined for an $\alpha$-large set of $\n$, but we can extend to all $\n$ by defining $F_{\n,m}$ to be the empty set for all other values of $\n$.)

By hypothesis, $\bigcap_{m=1}^M F_m$ is non-empty for any fixed $M$.  As a consequence, for each such $M$, we see that $\bigcap_{m=1}^M F_{\n,m}$ is non-empty for all $\n$ in an $\alpha$-large subset $S_M$ of $\N$.  By shrinking the $S_M$ if necessary, we may assume that they are decreasing in $M$.  For each $\n \in S_1$, let $M_\n$ be the largest natural number less than or equal to $\n$ with the property that $\n \in S_{M_\n}$, so that $\bigcap_{m=1}^{M_\n} F_{\n,m}$ is non-empty.  By the axiom of choice, we may thus find a sequence $(x_\n)_{\n \in S_1}$ such that $x_\n \in \bigcap_{m=1}^{M_\n} F_{\n,m}$ for all $\n \in S_1$.  If we form $x := \lim_{\n \to \alpha} x_\n$, then we have $x \in F_M$ for all $M$, since $x_\n \in F_{\n,M}$ for all $\n \in S_M$.  The claim follows.

Now let $f: X \to Y$ be an internal function.  It is clear that $f$ is continuous and open.  To demonstrate that it is closed, let $F$ be a countably closed subset of $X$, thus $F = \bigcap_{n=1}^\infty F_n$ for some internal subsets $F_n$ of $X$.  Observe that if $y \in Y$ lies in $f(F_n)$ for each $n$, then the internal sets $\{ x \in F_n: f(x) = y \}$ have all finite intersections non-empty, and hence by countable compactness $\{ x \in F: f(x) = y\}$ is non-empty as well.  This shows that $f(F) = \bigcap_{n=1}^\infty f(F_n)$, and so $f$ is closed as required.
\end{proof}

Henceforth we endow all internal sets with the $\sigma$-topological structure given by Lemma \ref{countsat}.  This structure is unfortunately not a genuine topology, as all points are internal and thus countably open, but arbitrary unions of points need not be countably open.  However, it turns out in practice that $\sigma$-topological structure can serve as a reasonable substitute for genuine topological structure, so long as one restricts attention to at most countably many sets at any given time (and provided that one works exclusively with sequences rather than with nets).

We also note the pleasant fact that the standard part function $\st: {}^* \R \to \R \cup \{\infty\}$ is a continuous map from ${}^*\R$ (with the $\sigma$-topological structure) to $\R \cup \{\infty\}$ (with the usual topological structure), thus the inverse image of an open (resp. closed) set in $\R \cup \{\infty\}$ is countably open (resp. countably closed) in ${}^* \R$. 

Given a finite non-empty set $X$, we can define the uniform probability measure $\mu_X$ on $X$ by the formula
$$ \mu_X(E) := |E|/|X|.$$
It turns out that this simple measure construction can be extended to ultraproducts of finite non-empty sets as well, and is known as the \emph{Loeb measure construction}:

\begin{definition}[Loeb measure]\label{loeb}\cite{loeb}  Let $X = \prod_{\n \to \alpha} X_\n$ be an ultraproduct of finite non-empty sets $X_\n$, and let $\mu_{X_\n}$ be the uniform probability measures on each $X_\n$.  Let ${\mathcal B}_X^0$ be the Boolean algebra of internal subsets of $X$, and let ${\mathcal B}_X$ be the $\sigma$-algebra generated by ${\mathcal B}_X^0$.  We define the \emph{Loeb measure} on $X$ to be the unique probability measure on ${\mathcal B}_X$ with the property that
\begin{equation}\label{mux}
 \mu_X(F) := \operatorname{st}( {}^* \mu_X(F) )
\end{equation}
whenever $F = \prod_{\n \to \alpha} F_\n$ is an internal subset of $X$, where 
$$
{}^* \mu_X(F) := \lim_{\n \to \alpha} \mu_{X_\n}( F_\n ) \in {}^* [0,1]$$
is the internal measure of $F$.
\end{definition}

To verify that Loeb measure actually exists and is unique, we observe from Lemma \ref{countsat} that the function $\mu_X$ defined on ${\mathcal B}_X^0$ is a premeasure with total mass one, and the claim then follows from the Carath\'eodory extension theorem (or, more precisely, the Hahn-Kolmogorov extension theorem).  

\begin{remark} One can view ${\mathcal B}_X^0$ as analogous to the algebra of elementary subsets of a Euclidean space (i.e. Boolean combinations of finitely many boxes), with ${\mathcal B}_X$ as analogous to the $\sigma$-algebra of Borel sets (indeed, note that this $\sigma$-algebra is generated by the countably open sets).  One could, if one wished, form the completion of Loeb measure by adjoining all sub-null sets to ${\mathcal B}_X$, thus giving a measure analogous to Lebesgue measure rather than Borel measure.  It will however be convenient to avoid working with this completion, as it has poorer properties with respect to restriction to measure zero subsets.  (This is analogous to how a slice of a Borel measurable subset of $\R^2$ is automatically Borel measurable in $\R$, whereas the analogous claim for Lebesgue measurable subsets is certainly false.)  

It will also be very important to keep in mind that on the product 
$$X \times Y = (\lim_{\n \to \alpha} X_\n) \times (\lim_{\n \to \alpha} Y_\n) \equiv \lim_{\n \to \alpha} (X_\n \times Y_\n)$$ 
of two ultraproducts $X = \lim_{\n \to \alpha} X_\n$, $Y = \lim_{\n \to \alpha} Y_\n$ of finite non-empty sets, the Loeb measure $\mu_{X \times Y}$ is \emph{not}, in general, the product $\mu_X \times \mu_Y$ of the Loeb measures on $X$ and $Y$; instead, the latter measure is a restriction of the former to a much smaller $\sigma$-algebra (${\mathcal B}_X \times {\mathcal B}_Y$ is usually much smaller than ${\mathcal B}_{X \times Y}$).  In a similar spirit, the $\sigma$-topology on $X \times Y$ is \emph{not} the product of the $\sigma$-topologies on $X$ and $Y$ in general, but is usually a much finer $\sigma$-topology.  We will discuss this important phenomenon in more detail later in this section.
\end{remark}

We record the following pleasant approximation property:

\begin{lemma}[Approximation by internal sets]\label{appx}  Let $X$ be an ultraproduct of finite non-empty sets, and let $E$ be a Loeb measurable subset of $X$.  Then there exists an internal subset $E'$ of $X$ that differs from $E$ by a $\mu_X$-null set (thus $\mu_X(E \Delta E') = 0$).
\end{lemma}

\begin{proof}  As the Loeb $\sigma$-algebra is generated by the Boolean algebra of internal subsets, it suffices to show that the property of differing from an internal subset by a $\mu_X$-null set is closed under complements and countable unions.  The complementation property is clear. To prove the countable union property, it suffices by countable additivity of $\mu_X$ to show that any countably open set $\bigcup_n E_n$ (where each $E_n$ is an internal subset of $X$) itself differs from an internal subset by a $\mu_X$-null set.  We may of course assume that the $E_n$ are disjoint.  Let $p$ denote the Loeb measure of $\bigcup_n E_n$, thus $p = \sum_{n=1}^\infty \mu(E_n)$.  For any $n_0 \in \N$, we clearly can find an internal subset $S$ of $X$ which contains $\bigcup_{n=1}^{n_0} E_n$ and has internal measure at most $p + \frac{1}{n_0}$; indeed, one can just take $S = \bigcup_{n=1}^{n_0} E_n$ itself.  This is an internal property of the set $S$, and so by countable saturation (applied to the internal power set $\prod_{\n \to \alpha} 2^{X_\n}$ of $X$) we conclude that there exists an internal subset $S$ of $X$ which contains $\bigcup_{n=1}^{n_0} E_n$ and has internal measure (and hence Loeb measure) at most $p + \frac{1}{n_0}$ for \emph{every} $n_0 \in \N$.  In particular, it can only differ from $\bigcup_{n=1}^\infty E_n$ by a $\mu_G$-null set, and the claim follows.
\end{proof}

One of the basic theorems in ordinary topology is Tychonoff's theorem that the arbitrary product of compact topological spaces is still compact.  Related to this is the assertion (proven using the Kolmogorov extension theorem) that the product of arbitrarily many (inner regular) probability spaces is still a probability space.  We now develop analogues of these two assertions for $\sigma$-topological spaces and for Loeb measure, which will be needed later in the paper when we wish to apply probability theory to a sequence of random variables drawn independently and uniformly at random from an ultraproduct of finite non-empty sets.

We first give the $\sigma$-topological version of Tychonoff's theorem, a fact closely related to the property of \emph{$\omega_1$-saturation} considered by model theorists.  

\begin{lemma}[$\sigma$-Tychonoff theorem]\label{sigma-tych}
Let $(X_a)_{a \in A}$ be a family of sets $X_a$ indexed by an at most countable set $A$, let $X_A := \prod_{a \in A} X_a$ be the product space, and for each $I \subset J \subset A$, let $\pi_{I \leftarrow J}: X_J \to X_I$ be the obvious projection map between the spaces $X_I := \prod_{a \in I} X_a$ and $X_J := \prod_{a \in J} X_a$.  Suppose that for each finite subset $I$ of $A$, $X_I$ is equipped with a countably compact $T_1$ $\sigma$-topology ${\mathcal F}_I$, such that the projection maps $\pi_{I \leftarrow J}$ are all both continuous and closed.  Define a \emph{cylinder set} on $X_A$ to be a set of the form $\pi^{-1}_{I \leftarrow A}(V_I)$, where $V_I$ is a countably open subset of $X_I$, and let ${\mathcal F}_A$ be the collection of all at most countable unions of cylinder sets.  Then $(X_A, {\mathcal F}_A)$ is also a countably compact $T_1$ $\sigma$-topological space.
\end{lemma}

\begin{proof}   It is clear that $(X_A, {\mathcal F}_A)$ is a $\sigma$-topological space, so we only need to verify countable compactness.  As in Lemma \ref{countsat}, it suffices to show that if $(E_n)_{n=1}^\infty$ is a sequence of cylinder sets $E_n = \pi_{I_n \leftarrow A}^{-1}(F_{I_n})$, where each $I_n$ is finite and $F_{I_n}$ is countably closed, with $\bigcap_{n=1}^M E_n$ non-empty for every finite $M$, then $\bigcap_{n=1}^\infty E_n$ is non-empty as well.

By increasing each $I_n$ if necessary (and using the continuity of the projection maps $\pi_{I \leftarrow J}$) we may assume that the $I_n$ are increasing in $I$, and then by shrinking the $F_{I_n}$ we may assume that $F_{I_m} \subset \pi_{I_n \leftarrow I_m}^{-1}(F_{I_n})$ for all $m \geq n$.  

We now recursively construct a sequence of points $x_n \in \bigcap_{I_n \leftarrow I_m}(F_{I_m})$ for $n=1,2,\ldots$ as follows.  To construct $x_1$, observe from the closed nature of the $\pi_{I \leftarrow J}$ that $\pi_{I_1 \leftarrow I_m}(F_{I_m})$ are countably closed non-empty decreasing subsets of $X_{I_1}$, hence by countable compactness $\bigcap_{m=1}^\infty \pi_{I_1 \leftarrow I_m}(F_{I_m})$ is non-empty.  We select a point $x_1$ arbitrarily from this set.  Now assume inductively that $n>1$ and that $x_{n-1}$ has already been constructed.  Then the sets $\pi_{I_n \leftarrow I_m}(F_{I_m} \cap \pi_{I_{n-1} \leftarrow I_m}^{-1}(\{x_{n-1}\}))$ for $m \geq n$ are countably closed non-empty decreasing subsets of $X_{I_1}$ (here we use the continuous and closed nature of the $\pi_{I \leftarrow J}$, as well as the $T_1$ nature of $I_{n-1}$), and hence by countable compactness we can find $x_n$ in the joint intersection $\bigcap_{m=n}^\infty \pi_{I_n \leftarrow I_m}(F_{I_m} \cap \pi_{I_{n-1} \leftarrow I_m}^{-1}(\{x_{n-1}\}))$ is non-empty.  By construction, we have $\pi_{I_n \leftarrow I_m}(x_m) = x_n$ for all $m \geq n$.  If we then select $x \in X_A$ such that $\pi_{I_n \leftarrow A}(x) = x_n$ for all $n$, we conclude that $x$ lies in every $F_n$ as required.
\end{proof}

Now we turn to product Loeb measures.  Let us first consider the problem of multiplying together two Loeb probability spaces $(X, {\mathcal B}_X, \mu_X)$ and $(Y, {\mathcal B}_Y, \mu_Y)$, where $X = \prod_{\n \to \alpha} X_\n, Y = \prod_{\n \to \alpha} Y_\n$ are ultraproducts of finite non-empty sets.  We have two probability space structures on the product $X \times Y$, namely the Loeb space $(X \times Y, {\mathcal B}_{X \times Y}, \mu_{X \times Y})$ and the product space $(X \times Y, {\mathcal B}_X \times {\mathcal B}_Y, \mu_X \times \mu_Y)$.  It is easy to see that the latter space is a restriction of the former, thus ${\mathcal B}_X \times {\mathcal B}_Y \subset {\mathcal B}_{X \times Y}$, and $\mu_X \times \mu_Y(E) = \mu_{X \times Y}(E)$ whenever $E \in {\mathcal B}_X \times {\mathcal B}_Y$.  This is ultimately because the Cartesian product of two internal sets is again an internal set (identifying Cartesian products of ultraproducts with ultraproducts of Cartesian products in the obvious manner).   On the other hand, not every set which is measurable in ${\mathcal B}_{X \times Y}$ is measurable in ${\mathcal B}_X \times {\mathcal B}_Y$; intuitively, the reason for this is that internal subsets of $X \times Y$ need not be approximable by Boolean combinations of finitely many Cartesian products of internal subsets of $X$ and $Y$ (or, at the finitary level, not all subsets of $X_\n \times Y_\n$ can be well approximated by Boolean combinations of finitely many subsets of $X_\n$ and $Y_\n$, where the number of such subsets is bounded uniformly in $\n$).  Thus, one should view the probability space $(X \times Y, {\mathcal B}_{X \times Y}, \mu_{X \times Y})$
as a strict \emph{extension} of the probability space $(X \times Y, {\mathcal B}_X \times {\mathcal B}_Y, \mu_X \times \mu_Y)$, or equivalently one should view the latter space as a strict factor of the former.

Despite the disparity between the two factors, we still have the following version of the Fubini-Tonelli theorem:

\begin{theorem}[Fubini-Tonelli theorem for Loeb measure]\label{ftl}  Let $X = \prod_{\n \to \alpha} X_\n, Y = \prod_{\n \to \alpha} Y_\n$ be ultraproducts of finite non-empty sets.  Let $f$ be a bounded ${\mathcal B}_{X \times Y}$-measurable function.  Then, for every $x \in X$, the function $y \mapsto f(x,y)$ is ${\mathcal B}_Y$-measurable, and the function $x \mapsto \int_Y f(x,y)\ d\mu_Y(y)$ is ${\mathcal B}_X$-measurable.  Furthermore, we have the identity
$$ \int_{X \times Y} f(x,y)\ d\mu_{X \times Y}(x,y) = \int_X \left(\int_Y f(x,y)\ d\mu_Y(y)\right)\ d\mu_X(x).$$
Similarly with the roles of $X$ and $Y$ reversed.  As a particular corollary, if $E$ is a $\mu_{X \times Y}$-null set in $X \times Y$, then for $\mu_X$-almost every $x \in X$, the set $E_x := \{ y \in Y: (x,y) \in E \}$ is a $\mu_Y$-null set, and similarly with the roles of $X$ and $Y$ reversed.
\end{theorem}

As with the usual Fubini-Tonelli theorem, one can generalize this theorem from bounded functions to non-negative or absolutely integrable functions (after excluding some null set of $X$ where the $Y$ integral is infinite or divergent), but we will not need to do so here.

\begin{proof}  This will be a slight variant of the usual proof of the Fubini-Tonelli theorem.  By approximating $f$ by simple functions and using linearity, we may reduce to the case when $f = 1_E$ is an indicator function for some $E \in {\mathcal B}_{X \times Y}$; thus our task is now to show that the slices $E_x := \{ y \in Y: (x,y) \in E\}$ are ${\mathcal B}_Y$-measurable for every $x$, the function $x \mapsto \mu_Y(E_x)$ is ${\mathcal B}_X$-measurable, and that
$$ \mu_{X \times Y}(E) = \int_X \mu_Y(E_x)\ d\mu_X(x).$$
By the monotone class lemma, it suffices to show that the set of $E$ in ${\mathcal B}_{X \times Y}$ obeying these conclusions is closed under upward unions, downward intersections, and contains the algebra ${\mathcal B}_{X \times Y}^0$ of internal subsets of $X \times Y$.  The first two claims follow from several applications of the monotone convergence theorem in the three probability spaces $(X \times Y, {\mathcal B}_{X \times Y}, \mu_{X \times Y})$, $(X, {\mathcal B}_X, \mu_X)$, and $(Y, {\mathcal B}_Y, \mu_Y)$.  So we may assume that $E$ is an internal subset of $X \times Y$, thus $E = \prod_{\n \to \alpha} E_\n$ where $E_\n \subset X_\n \times Y_\n$ for an $\alpha$-large set of $\n$.  Since
$$ E_x = \prod_{\n \to \alpha} (E_\n)_{x_\n}$$
whenever $x = \lim_{\n \to \alpha} x_n$, we see that $E_x$ is internal and thus ${\mathcal B}_Y$-measurable for all $x \in X$.  Also, from \eqref{mux} we have $\mu_Y(E_x) = \operatorname{st}( g(x) )$, where $g$ is the internal function
$$ g(x) := {}^* \mu_Y(E_x) = \lim_{\n \to \alpha} \frac{1}{|Y_\n|} |\{ y_\n \in Y_\n: (x_\n,y_\n) \in E_\n \}|.$$
From the trivial Fubini-Tonelli theorem for finite sets, we have
$$ \frac{1}{|X_\n|} \sum_{x_\n \in X_\n} g(x_\n) = \mu_{X_\n \times Y_\n}(E_\n).$$
Taking ultralimits, and approximating $g$ from above and below by simple functions, we see that
$$ \int_X g(x)\ d\mu_X(x) = \mu_{X \times Y}(E)$$
and the claim follows.
\end{proof}

We can now construct a Loeb product space with infinitely many factors as follows.

\begin{theorem}[Loeb product spaces]\label{prod}  Let $A$ be an index set (possibly countable or even uncountable), and let $(X_a)_{a \in A}$ be a family of internal sets $X_a$, with each $X_a$ being the ultraproduct of finite non-empty sets.  Let $X_A := \prod_{a \in A} X_a$ be the product space, and for each finite subset $I$ of $A$, let $\pi_I: X_A \to X_I$ be the projection to the space $X_I := \prod_{a \in I} X_a$.  Let ${\mathcal B}_{X_A}$ denote the $\sigma$-algebra generated by the pullbacks $\pi_I^{-1}({\mathcal B}_{X_I}) := \{ \pi_I^{-1}(E_I): E_I \in {\mathcal B}_{X_I} \}$ for all finite subsets $I$ of $A$; equivalently, ${\mathcal B}_{X_A}$ is generated by the cylinder sets from Lemma \ref{sigma-tych}.  Then there exist a unique probability measure $\mu_{X_A}$ on ${\mathcal B}_{X_A}$ whose pushforward measures $(\pi_I)_* \mu_{X_A}$ agree with $\mu_{X_I}$ for each finite subset $I$ of $A$, thus
\begin{equation}\label{muxa}
 \mu_{X_A}( \pi_I^{-1}( E_I ) ) = \mu_{X_I}(E_I)
\end{equation}
for all $E_I \in {\mathcal B}_{X_I}$.
\end{theorem}

\begin{proof}  By Lemma \ref{countsat} and Lemma \ref{sigma-tych}, the cylinder sets on $X_A$ generate a countably compact $T_1$ $\sigma$-topology on $X_A$.  Hence the function $\mu_{X_A}$ defined on the boolean algebra of cylinder sets by \eqref{muxa} is a premeasure of total mass $1$.  The claim then follows from the Carath\'eodory extension theorem (or Hahn-Kolmogorov extension theorem).
\end{proof}

\begin{remark}\label{prada}
We will refer to the probability space $(X_A, {\mathcal B}_{X_A}, \mu_{X_A})$ generated by Theorem \ref{prod} as the \emph{Loeb product space} on $X_A$.  In general, it will be a strict extension of the usual product probability space $(X, \prod_{a \in A} {\mathcal B}_a, \prod_{a \in A} \mu_a)$.  If we use this Loeb product space as the sample space for probability theory, then the coordinate projections from $X_A$ to each factor space $X_a$ can be interpreted as a family $(x_a)_{a \in A}$ of random variables, with each finite subtuple $(x_a)_{a \in I}$ being distributed with the law of the Loeb probability space $(X_I, {\mathcal B}_I, \mu_I)$.  In particular, for any ${\mathcal B}_{X_I}$-measurable set $E_I$, we have
$$ \P( (x_a)_{a \in I} \in E_I) = \mu_{X_I}( E_I ).$$
This property is stronger than joint independence of the $x_a$, because the $\sigma$-algebra ${\mathcal B}_{X_I}$ is significantly larger than the product $\sigma$-algebra $\prod_{a \in I} {\mathcal B}_{X_a}$.  
\end{remark}

\section{Proof of Theorem \ref{loaded}}\label{lo-sec}.

We now prove Theorem \ref{loaded}, whose proof is simpler than that of Theorem \ref{main}, but already illustrates the key strategies used in that latter proof, in particular the use of an ultraproduct correspondence principle to reduce the problem to an ergodic theoretic one.  More precisely, we will deduce Theorem \ref{loaded} from the following recurrence result:

\begin{proposition}[Double recurrence]\label{p1}  Let $G$ be an infinite (and possibly uncountable) group that acts on a probability space $(X,\mu)$ by two commuting measure-preserving actions $(L_g)_{g \in G}, (R_g)_{g \in G}$ (thus $\mu(L_g E) = \mu(R_g E) = \mu(E)$, $L_g L_h = L_{gh}$, $R_g R_h = R_{gh}$, and $L_g R_h = L_h R_g$ for all $g,h \in G$ and measurable $E \subset X$), and let $A$ be a subset of $X$ of positive measure.  Then there exists a non-trivial group element $g \in G$ such that
$$ \mu(A \cap L_g A \cap L_g R_g A) > 0.$$
In particular, $A \cap L_g A \cap L_g R_g A$ is non-empty.
\end{proposition}

\begin{proof} Since every infinite group contains a countably infinite subgroup, we may assume without loss of generality that $G$ is countable.  The claim now follows from \cite[Theorem 1.5]{berg4} (note carefully that this result does \emph{not} require $G$ to be amenable).  In fact, that theorem yields the stronger result that there exists $\lambda>0$ for which the set $\{ g: \mu(A \cap L_g A \cap L_g R_g A) > \lambda \}$ is both left-syndetic and right-syndetic (and is even central* and inverse central*; see \cite[p. 1256]{berg4} for definitions).   
\end{proof}

Now we can begin the proof of Theorem \ref{loaded}.  In order to emphasise the relationship with the ergodic theorem in Proposition \ref{p1}, we introduce the uniform probability measure $\mu_G$ on a finite group $G$, thus
$$ \mu_G(E) := |E|/|G|$$
for all $E \subset G$, and $L^2(G) = L^2(G,\mu_G)$.  We also introduce the left-shift $L_g$ and right-shift $R_g$ actions on $G$ by the formulae
\begin{equation}\label{lager}
 L_g x := gx; R_g x := xg^{-1};
\end{equation}
these are commuting actions of $G$ on itself.  They induce the associated Koopman operators on $L^2(G)$ by the formulae
$$ L_g f(x) := f(g^{-1} x)$$
and
$$ R_g f(x) := f(x g),$$
which are of course unitary operators.  Observe that for any $g \in G$, we have
$$ \mu_G( \{ x \in G: x, xg, gx \in A \} ) = \mu_G( A \cap L_g A \cap L_g R_g A )$$
and so our task is to show that for any $\delta>0$ there exist $N,\eps>0$ such that if $G$ is a finite group with $|G| \geq N$ and $A$ is a subset of $G$ with $\mu_G(A) \geq \delta$, then there exists a non-trivial $g \in G$ such that
$$ \mu_G( A \cap L_g A \cap L_g R_g A ) \ge \eps.$$

Suppose for sake of contradiction that the claim failed.  Carefully negating the quantifiers, and using the axiom of choice, we may then find a $\delta > 0$, a sequence $G_\n$ of finite groups for each $\n \in \N := \{1,2,3,\ldots\}$ with $|G_\n| \geq \n$, and subsets $A_\n$ of $G_\n$, with the properties that
\begin{equation}\label{mud}
 \mu_{G_\n}(A_\n) \geq \delta
 \end{equation}
and
\begin{equation}\label{rga}
 \mu_{G_\n}( A_\n \cap L_{\n,g_\n} A_\n \cap L_{\n,g_{\n}} R_{\n,g_{\n}} A_\n ) \leq \frac{1}{\n}
 \end{equation}
for all non-trivial $g_\n \in G_\n$, where we use $L_{\n,g_\n}, R_{\n,g_\n}$ to denote the left and right actions for $G_\n$. 

Fix all the above data $G_\n, A_\n, \delta$, which we can view as a sequence of finitary ``approximate counterexamples'' to Theorem \ref{loaded}.  The next step is to use an ultraproduct construction to pass from this sequence of approximate counterexamples to a genuine counterexample to Theorem \ref{loaded} and obtain the desired contradiction.  Versions of this ``compactness and contradiction'' strategy of course appear in many arguments, including some versions of the Furstenberg correspondence principle (see e.g. \cite{berg5}, \cite{berg6}); see also \cite{szegedy} for a construction closely related to the one used here. One could also formalize the arguments here in the language of nonstandard analysis, but we will avoid doing so in this paper in order to reduce the possibility of confusion.

As in the previous section, we fix a non-principal ultrafilter $\alpha \in \beta \N \backslash \N$ on the natural numbers.  We may now form the ultraproducts $G := \prod_{\n \to \alpha} G_\n$ and $A := \prod_{\n \to \alpha} A_\n$.  As the $G_\n$ were all groups, the ultraproduct $G$ is also a group (with the group and inversion operations being the ultralimits of the associated operations on $G_\n$).  On the other hand, as each $G_\n$ had cardinality at least $\n$, we easily verify that $G$ has cardinality at least $\n$ for each $\n$, so that $G$ is now an infinite group (indeed, it will necessarily uncountable, since it is countably compact thanks to Lemma \ref{countsat}).  The set $A$ is of course a subset of $G$.  Our objective is to use this data to build a counterexample to Theorem \ref{loaded}.

Let $(X,{\mathcal X},\mu) := (G, {\mathcal B}_G, \mu_G)$ be the Loeb probability space associated to $G$.  (We will use two different symbols $G,X$ to describe the same object here, in order to emphasise the conceptual distinction between the underlying space $X$, and the group $G$ that acts on that space.)
It is easy to see that the left and right actions $(L_g)_{g \in G}, (R_g)_{g \in G}$ are measure-preserving actions on $X$, thus for any $g \in G$, the maps $E \mapsto L_g E$ and $E \mapsto R_g E$ are measure-preserving on $(X,{\mathcal X},\mu)$.  Note however that $(g,x) \mapsto L_g x$ and $(g,x) \mapsto R_g x$ are jointly measurable as maps from $G \times X$ to $X$ only if one uses the Loeb product $\sigma$-algebra ${\mathcal B}_{G \times X}$ on $G \times X$, rather than the product Loeb $\sigma$-algebra ${\mathcal B}_G \times {\mathcal B}_X$.  To avoid technical issues associated to this, we will not perform any operation (e.g. integration in $G$ rather than in $X$) that would require joint measurability of the actions.

\begin{remark}  The Loeb probability measure $\mu_G$ on the group $G$ is closely analogous to a Haar probability measure on a compact group, with the main difference being that $G$ is only a (countably) compact group with respect to a $\sigma$-topological structure, rather than a genuinely topological structure.  (For instance, it is easy to see that the group operations $g \mapsto g^{-1}$, $(g,h) \mapsto gh$ are continuous with respect to the $\sigma$-topological structures on $G$ and $G \times G$.)
\end{remark}

The set $A$ is an internal subset of $X$ and is hence measurable in $(X,{\mathcal X},\mu)$. From \eqref{mud}, \eqref{mux} we see that
$$ \mu(A) \geq \delta;$$
in particular, $A$ has positive measure.  Applying Proposition \ref{p1}, we can thus find a non-trivial element $g$ of $G$ such that
$$ \mu(A \cap L_g A \cap L_g R_g A) > \eps$$
for some $\eps>0$.  Now, if we write $g = \lim_{\n \to \alpha} g_\n$, then $g_\n$ is non-trivial for an $\alpha$-large set of $\n$.
Furthermore, from \eqref{mux} we have
$$ \mu(A \cap L_g A \cap L_g R_g A) = \st \lim_{\n \to \alpha} \mu_{G_\n}(A_\n \cap L_{\n,g_\n} A_\n \cap L_{\n,g_\n} R_{\n,g_\n} A_\n) $$
and thus
$$ \mu_{G_\n}(A_\n \cap L_{\n,g_\n} A_\n \cap L_{\n,g_\n} R_{\n,g_\n} A_\n) > \eps$$
for an $\alpha$-large set of $\n$.  But this contradicts \eqref{rga} (note that as $\alpha$ is non-principal, any $\alpha$-large subset of $\N$ contains arbitrarily large elements).  This contradiction concludes the proof of Theorem \ref{loaded}. $\Box$

\begin{remark}  It is also possible to replace the use of ultraproducts in the above argument with applications of the compactness and completeness theorems in first-order logic instead, to obtain a countably saturated model\footnote{A model is \emph{countably saturated} if, whenever one has a countable family of sentences $S_1,S_2,\ldots$ with the property that any finite number of these sentences are simultaneously satisfiable, then the entire family is simultaneously satisfiable; this property is the model-theoretic analogue of Lemma \ref{countsat}.} $(G,A,\mu,L,R)$ of a group $G$ and set $A$ that formally lies in a space $X$ with a probability measure $\mu$ and commuting actions $(L_g)_{g \in G}, (R_g)_{g \in G}$, such that this model is a limit of the finitary models $(G_\n,A_\n,\mu_{G_\n}, L_\n, R_\n)$ in the sense that any statement in first-order logic which holds  for all but finitely many of the finitary models, is also true in the countably saturated model.  We leave the details to the interested reader.
\end{remark}

\section{An ergodic theorem}

We now begin the proof of Theorem \ref{main}, which follows a similar strategy to that of Theorem \ref{loaded} but with some additional complications.  In particular, we will need to replace the multiple recurrence theorem in Proposition \ref{p1} with a convergence theorem which, due to the inherent independence properties of quasirandom groups in our application, takes the shape of a relative weak mixing result along a properly chosen IP system.  To prove this theorem, it will be convenient to use the machinery of idempotent ultrafilters and their associated limits.

We turn to the details.  Given a group $G$, define an \emph{IP system} in $G$ to be a set\footnote{Strictly speaking, the IP system should be a tuple consisting of the set $H$ together with the generators $g_1,g_2,\ldots$, but we shall abuse notation and refer to the system simply by the set $H$.} of the form
$$ H = \{ g_{i_1} \ldots g_{i_r}: r \geq 1; 1 \leq i_1 < i_2 < \ldots < i_r \}$$
where $g_1,g_2,\ldots$ are an infinite sequence of elements in $G$ (not necessarily distinct).  Inside such an IP system, we can form the sub-IP system
$$ H_n = \{ g_{i_1} \ldots g_{i_r}: r \geq 1; n \leq i_1 < i_2 < \ldots < i_r \}$$
for any natural number $n$.  Given a sequence $(x_g)_{g \in H}$ of points in a Hausdorff topological space $Z$ indexed by $H$ and a point $x \in X$, we say that $x_g$ \emph{converges along $H$} to $x$, and write $H\!-\!\lim_g x_g = x$, if for every neighborhood $V$ of $x$, there exists $n$ such that $x_g \in V$ for all $g \in H_n$.

The variant of Proposition \ref{p1} that we will need is

\begin{theorem}\label{ergo}  Let $(X,{\mathcal X},\mu)$ be a probability space, and let $G$ be a group.  Let $(L_g)_{g \in G}$ and $(R_g)_{g \in G}$ be two measure-preserving actions of $G$ on $X$, which commute in the sense that $L_g R_h = R_h L_g$ for all $g,h \in G$.  Define $L_g f := f \circ L_g^{-1}$ and $R_g f := f \circ R_g^{-1}$ for $f \in L^\infty(X,\mu)$.  Let $H$ be an IP system in $G$, and $f_1, f_2, f_3$ be elements of $L^\infty(X,\mu)$.  Assume the following mixing properties:
\begin{itemize}
\item[(i)]  (Left mixing) For any $f,f' \in L^2(X,{\mathcal X},\mu)$, one has
$$ H\!-\!\lim_g \int_X f L_g f'\ d\mu = \left(\int_X f\ d\mu\right) \left(\int_X f'\ d\mu\right).$$
(In particular, we assume that this $H$-limit exists.)
\item[(ii)]  (Right mixing) For any $f,f' \in L^2(X,{\mathcal X},\mu)$, one has
$$ H\!-\!\lim_g \int_X f R_g f'\ d\mu = \left(\int_X f\ d\mu\right) \left(\int_X f'\ d\mu\right).$$
\item[(iii)]  ($f_3$ orthogonal to diagonally rigid functions)  Let $f \in L^2(X,{\mathcal X},\mu)$ be any function with the rigidity property that, for any $\eps>0$ and for any natural number $n$, there exists $g \in H_n$ such that $\| L_g R_g f - f \|_{L^2(X,{\mathcal X},\mu)} \leq \eps$.  Then $\int_X f f_3\ d\mu = 0$.
\end{itemize}
Then for any $\eps>0$ and natural number $n$, there exists $g \in H_n$ such that
$$ |\int_X f_1 (L_g f_2) (L_g R_g f_3)\ d\mu| \leq \eps.$$
\end{theorem}

Note that we allow $G$ to be uncountable.  However, observe that we may without loss of generality restrict $G$ to the group generated by $H$, so we may assume without loss of generality that $G$ is at most countable.  Note also that the left and right mixing properties are assumed to apply to \emph{all} $L^2$ functions $f,f'$, not just the three given functions $f_1,f_2,f_3$, but the diagonal mixing property (or more precisely, the orthogonality to diagonally rigid function property) is only imposed for the specific function $f_3$.  In our application, we cannot impose diagonal mixing for arbitrary functions, because of the non-ergodicity of the conjugation action (any subset of a group $G$ which is a union of conjugacy classes is clearly invariant with respect to conjugation).  

To prove Theorem \ref{ergo} we will use the tool\footnote{It should also be possible to prove Theorem \ref{ergo} without idempotent ultrafilters, by replacing the notion of a $p$-limit with that of an IP-limit.  But then one would need to repeatedly invoke Hindman's theorem \cite{hind} as a substitute for the idempotent property, which would require one to continually pass to IP subsystems of the original IP system.  We will not detail this approach to Theorem \ref{ergo} here.} of \emph{idempotent ultrafilters}.  We stress that despite some superficial similarities, these ultrafilters are unrelated to the non-principal ultrafilter $\alpha$ used in the ultraproduct correspondence principle, and are used for completely different purposes.

We first recall the definition of an idempotent ultrafilter.  See \cite{berg1,berg2,berg3} for some surveys on idempotent ultrafilters and their uses in ergodic theory.

\begin{definition}[Idempotent ultrafilter] Let $G$ be an at most countable group.  Given an ultrafilter $p \in \beta G$, define the product ultrafilter $p \cdot p$ by requiring that $A \in p \cdot p$ if and only if $Ag^{-1}$ is $p$-large for a $p$-large set of $g$.  (Recall that a subset of $G$ is \emph{$p$-large} if it lies in $p$.)  We say that the ultrafilter $p$ is \emph{idempotent} if $p \cdot p = p$.  
\end{definition}

We have the following basic existence theorem:

\begin{theorem}[IP systems support idempotent ultrafilters]\label{exist} Let $H$ be an IP system in an at most countable group $G$.  Then there exists an idempotent ultrafilter $p$ on $G$ supported by every $H_n$ (i.e. $H_n$ is $p$-large for all $n$).
\end{theorem}

\begin{proof} See \cite[Theorem 2.5]{berg3} (the proof there is stated for actions of $\N$, but the argument is general and applies to arbitrary groups or semigroups $G$).  In \cite{hindman}, this result is attributed to Galwin.
\end{proof}

We will need the notions of $p$-limit and $IP$-limit.
If $p$ is an ultrafilter on $G$, $(x_g)_{g \in H}$ is a tuple in a Hausdorff topological space $Z$ indexed by a $p$-large set $H$,  and $x$ is a point in $Z$, we say that $x_g$ \emph{converges along $p$} to $x$, and write $p\!-\!\lim_g x_g = x$, if for every neighborhood $V$ of $x$, the set $\{g\in H: x_g \in V \}$ lies in $p$.  

Note that if $H$ is an IP system, and $p$ is an ultrafilter that is supported by $H_n$ for every $n$, then convergence along $H$ implies convergence along $p$, thus if $(x_g)_{g \in H}$ in a Hausdorff topological space $Z$ indexed by $H$ and $x \in Z$, then if $H\!-\!\lim_g x_g = x$ then $p\!-\!\lim_g x_g = x$; conversely, if $p\!-\!\lim_g x_g = x$, then for every $n$ and every neighborhood $V$ of $x$ there exists $g \in H_n$ such that $x_g \in V$.  

Given a unitary action $(U_g)_{g \in G}$ of an at most countable group $G$ on a Hilbert space $W$, and an idempotent ultrafilter $p \in \beta G$, we say that an element $f$ of $W$ is \emph{rigid} with respect to the $(U_g)_{g \in G}$ action along $p$ if one has $p\!-\!\lim_g U_g f = f$ in the weak topology of $W$.  
We will need the following ergodic theorem for idempotent ultrafilters:

\begin{theorem}[Idempotent ergodic theorem]\label{iet}  Let $(U_g)_{g \in G}$ be a unitary action of an at most countable group $G$ on a Hilbert space $W$, and let $f \in W$, and let $p$ be an idempotent ultrafilter on $G$.  Then $p\!-\!\lim_g U_g f$ exists in the weak topology of $W$ and is equal to $Pf$, where $P$ is the orthogonal projection to the closed linear subspace $\{ f': p\!-\!\lim_g U_g f' = f' \}$ of $W$ consisting of functions that are rigid with respect to the $(U_g)_{g \in G}$ action along $p$.
\end{theorem}

\begin{proof}  See \cite[Theorem 2.4]{berg4}.  We remark that a related theorem (under the additional hypothesis that the idempotent ultrafilter $p$ is minimal) was established in \cite[Corollary 4.6]{berg2}.  In the minimal idempotent case we also have the additional property that $P$ commutes with the $U_g$; see \cite[Theorem 2.4]{berg4}.  However, we will not need this additional fact in our arguments here.
\end{proof}

Finally, we will need the following version of the van der Corput lemma for idempotent ultrafilters.

\begin{theorem}[Idempotent van der Corput lemma]\label{vdc}  Let $(U_g)_{g \in G}$ be a unitary action of an at most countable group $G$ on a Hilbert space $W$, and let $p$ be an idempotent ultrafilter on $G$. If $(f_g)_{g \in G}$ is a bounded family of vectors in $W$ with the property that 
$$ p\!-\!\lim_h  p\!-\!\lim_g \langle f_{gh}, f_g \rangle_W = 0,$$
then
$$p\!-\!\lim_g f_g = 0$$
in the weak topology.  
\end{theorem}

\begin{proof}  See \cite[Theorem 2.3]{berg4}.
\end{proof}

We have enough machinery to prove Theorem \ref{ergo}.

\begin{proof} (Proof of Theorem \ref{ergo})  As discussed previously, we may assume without loss of generality that $G$ is at most countable.  By Theorem \ref{exist}, we may find an idempotent ultrafilter $p \in \beta G$ which is supported by $H_\n$ for every $\n$.
From hypothesis (i) we see that for all $f' \in L^2(X,{\mathcal X},\mu)$, we have
$$ H\!-\!\lim_g L_g f' = \int_X f'\ d\mu$$
in the weak topology of $L^2(X,{\mathcal X},\mu)$, and hence
\begin{equation}\label{plg}
 p\!-\!\lim_g L_g f' = \int_X f'\ d\mu
\end{equation}
in the weak topology of $L^2(X,{\mathcal X},\mu)$ also. Similarly, from (ii) we have
\begin{equation}\label{prg}
 p\!-\!\lim_g R_g f' = \int_X f'\ d\mu
\end{equation}
in the weak topology of $L^2(X,{\mathcal X},\mu)$ for all $f' \in L^2(X,{\mathcal X},\mu)$.  Next, if $f \in L^2(X,{\mathcal X},\mu)$ is rigid with respect to the $(L_g R_g)_{g \in G}$ action along $p$, thus
$$ p\!-\!\lim_g L_g R_g f = f,$$
in the weak topology, and in particular
$$ p\!-\!\lim_g \langle f, L_g R_g f\rangle_{L^2(X,{\mathcal X},\mu)} = \|f\|_{L^2(X,{\mathcal X},\mu)}^2;$$
on the other hand, by the parallelogram law and the unitary nature of $L_g R_g$ we have
$$ \| f - L_g R_g f \|_{L^2(X,{\mathcal X},\mu)}^2 = 2\|f\|_{L^2(X,{\mathcal X},\mu)}^2 - 2\langle f, L_g R_g f\rangle_{L^2(X,{\mathcal X},\mu)}$$
and thus
$$ p\!-\!\lim_g \| f - L_g R_g f \|_{L^2(X,{\mathcal X},\mu)} = 0$$
and so we have
$$ p\!-\!\lim_g L_g R_g f = f,$$
in the \emph{strong} topology also.  In particular,
$$ H\!-\!\lim_g L_g R_g f = f,$$
By hypothesis (iii), this forces $f$ to be orthogonal to $f_3$, thus we have
\begin{equation}\label{Rig}
\int_X f f_3\ d\mu = 0
\end{equation}
whenever $f \in L^2(X,{\mathcal X},\mu)$ is rigid with respect to the $(L_g R_g)_{g \in G}$ action along $p$.

To establish the theorem, it will suffice to show that
$$
 p\!-\!\lim_g \int_X f_1 (L_g f_2) (L_g R_g f_3)\ d\mu = 0,
$$
or equivalently that
\begin{equation}\label{pume}
 p\!-\!\lim_g (L_g f_2) (L_g R_g f_3)= 0
 \end{equation}
in the weak topology.

Let us first consider the case when $f_2 = 1$, that is we will show
\begin{equation}\label{pume-1}
 p\!-\!\lim_g \int_X f_1 (L_g R_g f_3)\ d\mu = 0.
\end{equation}
By Theorem \ref{iet}, the left-hand side of \eqref{pume-1} is equal to $\int_X f_1 P f_3$, where $P$ is the orthogonal projection onto the functions that are rigid with respect to the $(L_g R_g)_{g \in G}$ action along $p$; but from \eqref{Rig} we have $Pf_3 = 0$.  This establishes \eqref{pume-1}.

By linearity, we may now reduce to the task of establishing \eqref{pume} when $f_2$ has mean zero: 
\begin{equation}\label{f2x}
\int_X f_2\ d\mu = 0.  
\end{equation}
To handle this case, we apply Theorem \ref{vdc} with $W := L^2(X,{\mathcal X},\mu)$ and $f_g := (L_g f_2) (L_g R_g f_3)$, we see that to show \eqref{pume}, it will suffice to show that
\begin{equation}\label{cl}
p\!-\!\lim_h  p\!-\!\lim_g \int_X (L_{gh} f_2) (L_{gh} R_{gh} f_3) (L_g f_2) (L_g R_g f_3)\ d\mu= 0.
\end{equation}
We may rearrange the left-hand side (using the commutativity of the $L$ and $R$ actions) as
$$ 
p\!-\!\lim_h  p\!-\!\lim_g \int_X (f_2 L_h f_2) R_g (f_3 L_h R_h f_3)\ d\mu.$$
Applying \eqref{prg}, we may simplify this as
$$ 
p\!-\!\lim_h  (\int_X f_2 L_h f_2\ d\mu) (\int_X f_3 L_h R_h f_3\ d\mu).$$
From \eqref{plg}, \eqref{f2x} we have
$$ p\!-\!\lim_h  \int_X f_2 L_h f_2\ d\mu = (\int_X f_2\ d\mu) (\int_X f_2\ d\mu)  = 0.$$
Since $\int_X f_3 L_h R_h f_3\ d\mu$ is bounded in $h$, the claim \eqref{cl} follows.
\end{proof}

\begin{remark}  An inspection of the above argument reveals that we have actually proven an idempotent ultrafilter version of Theorem \ref{ergo}, in which the IP system $H$ is replaced by an idempotent ultrafilter $p$, the notion of an $H$-limit is replaced by a $p$-limit, the rigidity property in Theorem \ref{ergo}(iii) is replaced by the hypothesis that $p\!-\!\lim_g L_g R_g f = f$ (in the strong $L^2$ topology), and the conclusion is that $p\!-\!\lim_g \int_X f_1 (L_g f_2) (R_g f_3)\ d\mu = 0$.  
\end{remark}

\section{Ultra quasirandom groups}\label{ultrasec}

Throughout this section, we fix a non-principal ultrafilter $\alpha \in \beta \N \backslash \N$.

In order to prove Theorem \ref{main}, we will follow the proof of Theorem \ref{loaded} and perform an ultraproduct of a series of proposed counterexamples to Theorem \ref{main}.  In the course of doing so, we will be studying ultraproducts of increasingly quasirandom groups.  It will be convenient to give a name to such an object:

\begin{definition}[Ultra quasirandom group]  An \emph{ultra quasirandom group} is an ultraproduct $G = \prod_{\n \to \alpha} G_\n$ of finite groups with the property that for every $D>0$, the groups $G_\n$ are $D$-quasirandom for an $\alpha$-large set of $\n$.  (Informally: the quasirandomness of the $G_\n$ goes to infinity as $\n$ approaches $\alpha$.)
\end{definition}

To give an example of an ultra quasirandom group, we recall the following classical result of Frobenius:

\begin{lemma}[Frobenius]\label{frob}  Let $F$ be a finite field of prime order $p$, then the group $SL_2(F)$ of $2 \times 2$ matrices of determinant $1$ with entries in $F$ is $\frac{p-1}{2}$-quasirandom.
\end{lemma}

\begin{proof}  We may of course take $p$ to be odd.  Suppose for contradiction that we have a non-trivial representation $\rho: SL_2(F_p) \to U_d(\C)$ on a unitary group of some dimension $d$ with $d < \frac{p-1}{2}$.  Set $a$ to be the group element
$$ a := \begin{pmatrix} 1 & 1 \\ 0 & 1 \end{pmatrix},$$
and suppose first that $\rho(a)$ is non-trivial.  Since $a^p=1$, we have $\rho(a)^p=1$; thus all the eigenvalues of $\rho(a)$ are $p^{\operatorname{th}}$ roots of unity.  On the other hand, by conjugating $a$ by diagonal matrices in $SL_2(F_p)$, we see that $a$ is conjugate to $a^m$ (and hence $\rho(a)$ conjugate to $\rho(a)^m$) whenever $m$ is a quadratic residue mod $p$.  As such, the eigenvalues of $\rho(a)$ must be permuted by the operation $x \mapsto x^m$ for any quadratic residue mod $p$.  Since $\rho(a)$ has at least one non-trivial eigenvalue, and there are $\frac{p-1}{2}$ distinct quadratic residues, we conclude that $\rho(a)$ has at least $\frac{p-1}{2}$ distinct eigenvalues.  But $\rho(a)$ is a $d \times d$ matrix with $d < \frac{p-1}{2}$, a contradiction.  Thus $a$ lies in the kernel of $\rho$.  By conjugation, we then see that this kernel contains all unipotent matrices.  But these matrices are easily seen to generate $SL_2(F_p)$, and so $\rho$ is trivial, a contradiction.
\end{proof}

Thus, if $p_\n$ is any sequence of primes going to infinity, and $F$ is the pseudo-finite field $F := \prod_{\n \to \alpha} F_{p_\n}$, then $SL_2(F)$ will be an ultra quasirandom group.

Our plan for proving Theorem \ref{main} will be as follows.  First, we shall establish mixing properties for ultra quasirandom groups $G$, which provide control on expressions such as
$$ \int_G f L_g f'\ d\mu_G,$$
$$ \int_G f R_g f'\ d\mu_G,$$
or
$$ \| f - L_g R_g f \|_{L^2(G, {\mathcal B}_G, \mu_G)}^2$$
for $f,f' \in L^2(G, {\mathcal B}_G, \mu_G)$ and $\mu_G$-almost every $g \in G$.  Next, we will construct a \emph{random} IP system by selecting generators $g_1,g_2,\ldots$ uniformly at random from $G$ (this requires the Loeb product measure construction from \ref{prod}), and verify that the mixing properties described above are almost surely inherited by such an IP system.  Finally, we use the mixing properties of quasirandom groups one final time to show that we can construct a \emph{determinstic} IP system with the same properties, while also being contained inside a specified positive measure subset $E$ of $G$.  Using this IP system, Theorem \ref{ergo}, and an argument by contradiction, one can obtain an ultraproduct version of Theorem \ref{main}; and then by using {\L}os's theorem we obtain Theorem \ref{main} itself.

We turn to the details.  Let $G$ be an ultra quasirandom group, then we have the Loeb measure space $(G, {\mathcal B}_G, \mu)$.  The mixing property of finite quasirandom groups from Proposition \ref{wm} is then reflected in ultra quasirandom groups as follows:

\begin{lemma}[Weak mixing]\label{wmi}  Let $G$ be an ultra quasirandom group, and let $A, B \in {\mathcal B}_G$ be Loeb measurable subsets of $G$.  Then for $\mu_G$-almost every $g \in G$, we have
$$ \mu_G(A \cap L_g B) = \mu_G(A) \mu_G(B)$$
and
$$ \mu_G(A \cap R_g B) = \mu_G(A) \mu_G(B).$$
In a similar spirit, if $f, f' \in L^2(G, {\mathcal B}_G, \mu_G)$, we have for $\mu_G$-almost every $g \in G$ that
$$ \int_G f L_g f'\ d\mu_G = (\int_G f\ d\mu_G) (\int_G f'\ d\mu_G)$$
and
$$ \int_G f R_g f'\ d\mu_G = (\int_G f\ d\mu_G) (\int_G f'\ d\mu_G).$$
\end{lemma}

\begin{proof}  By approximating $L^2$ functions by simple functions, we see that the latter two conclusions are consequences of the former two.  We will just prove the first claim, as the second claim is similar. From Lemma \ref{appx} and a routine limiting argument, we see that to establish the lemma, it suffices to do so in the case when $A, B$ are internal sets, thus $A = \prod_{\n \to \alpha} A_\n$ and $B = \prod_{\n \to\alpha} B_\n$.

Let $\eps>0$ and $D>0$.  By hypothesis, $G_\n$ is $D$-quasirandom for $\n$ sufficiently close to $\alpha$.  By Proposition \ref{wm} and Markov's inequality, we have that
$$ |\mu_{G_\n}(A_\n \cap L_{g_\n} B_\n) - \mu_{G_\n}(A_\n) \mu_{G_\n}(B_\n)| \leq \eps$$
for all $g_\n$ in $G_\n \backslash E_\n$, where $E_\n$ is an exceptional set with $\mu_{G_\n}(E_\n) \leq \eps^{-1} D^{-1/2}$.  If we set $E := \prod_{\n \to \alpha} E_\n$, then on taking ultralimits we see that $E$ is an internal subset of $G$ with $\mu(E) \leq \eps^{-1} D^{-1/2}$, and that
$$ |\mu_G(A \cap L_g B) - \mu_{G}(A) \mu_{G}(B)| \leq \eps$$
for all $g$ outside of $E$.  Letting $D$ go to infinity, and then letting $\eps$ go to zero, we obtain the first conclusion of the lemma as desired.
\end{proof}

Let ${\mathcal I}_G$ be the sub-$\sigma$-algebra of ${\mathcal B}_G$ generated by the conjugation invariant functions (or sets).  We will need the following variant of the above mixing property, but for $L_g R_g$ instead of $L_g$ or $R_g$ separately:

\begin{lemma}[Almost sure relative diagonal mixing]\label{Drm-basic}  
Let $G$ be an ultra quasirandom group. Let $n$ be a natural number, and let $f: G \to [-1,1]$ be a Loeb measurable function.  Then for $\mu_G$-almost every $g \in G$, one has the identity 
\begin{equation}\label{doe}
\| f - L_g R_g f \|_{L^2(G, {\mathcal B}_G, \mu_G)}^2 =  2 \| f - \E(f|{\mathcal I}_G) \|_{L^2(G, {\mathcal B}_G, \mu_G)}^2.
\end{equation}
\end{lemma}

One can interpret \eqref{doe} geometrically as the assertion that $f - \E(f|{\mathcal I}_G)$ and $L_g R_g f - \E(f|{\mathcal I}_G)$ are orthogonal.

To prove Lemma \ref{Drm-basic}, we will need a technical relationship between the algebra ${\mathcal I}_G$ and the associated algebras ${\mathcal I}_{G_\n}$ of the finitary groups $G_\n$:

\begin{lemma}\label{conj}  For each $\n$, let $f_\n: G_\n \to [-1,1]$ be a function.  Then one has
$$
\E( (\st \lim_{\n \to \alpha} f_\n) | {\mathcal I}_G ) = \st \lim_{\n \to \alpha} \E(f_\n | {\mathcal I}_{G_\n} )$$
$\mu_G$-almost everywhere.
\end{lemma}

\begin{proof}  The function $\st \lim_{\n \to \alpha} \E(f_\n | {\mathcal I}_{G_\n} )$ is clearly invariant under conjugation by elements of $G$. It thus suffices to show that the function
$$ \tilde f := \st \lim_{\n \to \alpha} f_\n - \st \lim_{\n \to \alpha} \E(f_\n | {\mathcal I}_{G_\n} )$$
is orthogonal to all ${\mathcal I}_G$-measurable bounded functions.

Let $F: G \to [-1,1]$ be ${\mathcal I}_G$-invariant.  Then by conjugating $x$ by an arbitrary group element $h \in G$, we have
$$ \int_G F(x) \tilde f(x)\ d\mu_G(x) = \int_G F(x) \tilde f(hxh^{-1})\ d\mu_G(x).$$
Integrating in the $h$ variable and using the Fubini-Tonelli theorem (Theorem \ref{ftl}), we see that
$$ \int_G F(x) \tilde f(x)\ d\mu_G(x) = \int_G F(x) \left(\int_G \tilde f(hxh^{-1})\ d\mu_G(h)\right)\ d\mu_G(x).$$
It will thus suffice to show that
$$ \int_G \tilde f(hxh^{-1})\ d\mu_G(h) = 0$$
for any $x \in G$.  But we can write $\tilde f = \st \lim_{\n \to \alpha} \tilde f_\n$, where
$$ \tilde f_\n := f_\n - \E(f_\n | {\mathcal I}_{G_\n} )$$
and direct calculation shows that
$$ \int_{G_\n} \tilde f_\n( h_\n x_\n h_\n^{-1} )\ d\mu_{G_\n}(h_\n) = 0$$
for any $x_\n \in G_\n$, and the claim then follows from taking ultralimits (and approximating $\tilde f_\n$ above and below by simple functions).
\end{proof}

Now we can prove Lemma \ref{Drm-basic}.

\begin{proof}[Proof of Lemam \ref{Drm-basic}]  For each $g \in G$, define the quantity
$$ \delta(g) := \| f - L_g R_g f \|_{L^2(G,{\mathcal B}_G, \mu_G)}.$$
We can expand 
$$ \delta^2(g) = \int_G |f(x) - f(g^{-1} x g)|^2\ d\mu_G(x).$$
If we then let $f_x: G \to \R$ be the function
$$ f_x(h) := f(h^{-1} x h)$$
then we conclude (by Fubini-Tonelli or from consideration of the finitary case) that
$$ \delta^2(g) = \int_G \left(\int_G |f_x(h) - f_x(h g)|^2\ d\mu_G(h)\right)\ d\mu_G(x).$$
Expanding out the square, we obtain
$$ \delta^2(g) = 2\int_G \left(\|f_x\|_{L^2(G)}^2 - \int_G f_x(h) f_x(h g)\ d\mu_G(h)\right)\ d\mu_G(x).$$
By Lemma \ref{wmi}, we see that for each $x \in G$, we have
\begin{equation}\label{fxx}
 \int_G f_x(h) f_x(h g)\ d\mu_G(h) = \left(\int_G f_x\ d\mu_G\right)^2
\end{equation}
for $\mu_G$-almost every $g \in G$.  By the Fubini-Tonelli theorem (Theorem \ref{ftl}), we conclude that for $\mu_G$-almost every $g \in G$, one has \eqref{fxx} for $\mu_G$-almost every $x$.  

In particular, for $\mu_G$-almost every $g \in G$, we have
$$
\delta^2(g) = X$$
where
$$ X :=  2 \int_G (\|f_x\|_{L^2(G)}^2 - |\int_G f_x\ d\mu_G|^2 )\ d\mu_G(x).$$
We can of course write
$$ \|f_x\|_{L^2(G)}^2 - |\int_G f_x\ d\mu_G|^2 = \|f_x - \int_G f_x\ d\mu_G \|_{L^2(G)}^2.$$
At this point it is convenient to pass back to the finitary setting.  We can rewrite the previous formula for $X$ as
$$ X = \st \lim_{\n \to \alpha} 2 \E_{x \in G_\n} \| f_{x,\n} - \E_{G_\n} f_{x,\n} \|_{L^2(G_\n)}^2$$
where $f_{x,\n}(h) := f_\n(h^{-1} x h)$.
As all the fibers of the map $h \mapsto h^{-1} x h$ have the same cardinality, we have
$$ \E_{G_\n} f_{x,\n} = \E(f_\n|{\mathcal I}_{G_\n})(x).$$
Also we observe the identity
$$ \E_{x \in G_\n} \|f_{x,\n} - \E(f_\n|{\mathcal I}_{G_\n})(x)\|_{L^2(G_\n)}^2 = \| f_\n - \E(f_\n|{\mathcal I}_{G_\n})\|_{L^2(G_\n)}^2$$
which follows from the invariance of $\E(f_\n|{\mathcal I}_{G_\n})$ with respect to conjugations.
In summary, we conclude that
$$
\delta^2(g) = 2 \st \lim_{\n \to \alpha} \| f_\n - \E(f_\n|{\mathcal I}_{G_\n}) \|_{L^2(G_\n)}^2$$
for $\mu_G$-almost every $g \in G$, and the claim now follows from Lemma \ref{conj}.
\end{proof}

In what follows, we would like to introduce a sequence $g_1,g_2,g_3,\ldots$ of elements drawn uniformly and independently at random from the ultra quasirandom group $G$.  One can of course model this system of random variables rigorously by using the standard product space
$$ (G,{\mathcal B}_G,\mu_G)^\N = (G^\N, {\mathcal B}_G^\N, \mu_G^\N).$$
However, the product $\sigma$-algebra ${\mathcal B}_G^\N$ will be far too small to measure events involving products of two or more of the $g_i$, and will therefore be useless for our applications.  We will thus need to invoke Theorem \ref{prod} to construct an extension of the standard product space which can handle finite products of the $g_i$.  

We turn to the details.  Let $G$ be an ultra quasirandom group.  We construct a sequence $g_1,g_2,g_3,\ldots \in G$ of random variables in $G$, whose joint distribution $(g_a)_{a \in \N}$ is defined by the coordinate functions on the Loeb product space on $G^\N = \prod_{a \in \N} G$ as discussed in Remark \ref{prada}.  In particular, we have
\begin{equation}\label{powa}
 \P( (g_a)_{a=1}^k \in E) = \mu_{G^k}( E )
\end{equation}
for any natural number $k$ and any ${\mathcal B}_{G^k}$-measurable set $E$.  We can then form the random IP system
$$ H := \{ g_{i_1} \ldots g_{i_r}: r \geq 1; 1 \leq i_1 < i_2 < \ldots < i_r \}$$
and its sub-IP systems
$$ H_n := \{ g_{i_1} \ldots g_{i_r}: r \geq 1; n \leq i_1 < i_2 < \ldots < i_r \}$$
for any natural number $n$.

We now investigate the mixing properties of this random IP system.  First, we demonstrate almost sure IP-mixing of $H$ when applied to sets that only depend on finitely many of the generators of $H$:

\begin{lemma}[Almost sure left and right mixing]\label{lrm}  Let $G, (g_a)_{a \in \N}, H$ be as above.  Let $n$ be a natural number, and let $E = E_{g_1,\ldots,g_n}, F = F_{g_1,\ldots,g_n}$ be Loeb measurable subsets of $G$ that depend in some jointly Loeb measurable fashion on $g_1,\ldots,g_n$ (but do not depend on $g_{n+1}, g_{n+2},\ldots$), in the sense that the sets
$$ \{ (x,g_1,\ldots,g_n) \in G \times G^n: x \in E_{g_1,\ldots,g_n} \}$$
and
$$ \{ (x,g_1,\ldots,g_n) \in G \times G^n: x \in F_{g_1,\ldots,g_n} \}$$
are ${\mathcal B}_{G \times G^n}$-measurable.  Then almost surely, one has
$$ \mu_G( E \cap L_g F ) = \mu_G(E) \mu_G(F)$$
and
$$ \mu_G( E \cap R_g F ) = \mu_G(E) \mu_G(F)$$
for every $g \in H_{n+1}$.
\end{lemma}

\begin{proof}  As $H_{n+1}$ is at most countable, it suffices to verify the claim for $g = g_{i_1} \ldots g_{i_r}$ for a single choice of $r \geq 1$ and $n < i_1 \leq \ldots \leq i_r$.  By the Fubini-Tonelli theorem (Theorem \ref{ftl}), it suffices to prove the claim after replacing (or conditioning) the random variables $g_1,\ldots,g_n \in G$ with deterministic elements of $G$ (note that this does not affect the joint distribution of $g_{i_1},\ldots,g_{i_r}$, thanks to Fubini-Tonelli). But since $r \geq 1$, we see that $g = g_{i_1} \ldots g_{i_r}$ is uniformly distributed in $G$ (in the sense that it has the law of $\mu_G$ on ${\mathcal B}_G$, thus $\P( g \in A ) = \mu_G(A)$ for all $A \in {\mathcal B}_G$); this can be verified by reducing to the case of internal sets and then verifying the analogous finitary claim for products $g_\n = g_{i_1,\n} \ldots g_{i_r,\n}$ of uniformly distributed independent random variables $g_{i_1,\n},\ldots,g_{i_r,\n}$ on $G_\n$.  The claim now follows from Lemma \ref{wmi}.
\end{proof}

We have an analogue of the above lemma for the diagonal action $L_g R_g$:

\begin{lemma}[Almost sure relative diagonal mixing]\label{Drm}  
Let $G, (g_a)_{a \in \N}, H$ be as above.  Let $n$ be a natural number, and let $f = f_{g_1,\ldots,g_n}: G \to [-1,1]$ be a Loeb measurable function
that depends in a jointly Loeb measurable fashion on the random variables $g_1,\ldots,g_n$ (but do not depend on $g_{n+1}, g_{n+2},\ldots$), in the sense that the function $(x,g_1,\ldots,g_n) \mapsto f_{g_1,\ldots,g_n}(x)$ is a measurable function from $G \times G^n$ to $[-1,1]$.  Then almost surely, one has the identity 
\begin{equation}\label{doe-2}
\| f - L_g R_g f \|_{L^2(G, {\mathcal B}_G, \mu_G)}^2 =  2 \| f - \E(f|{\mathcal I}_G) \|_{L^2(G, {\mathcal B}_G, \mu_G)}^2
\end{equation}
for all $g \in H_{n+1}$.
\end{lemma}

\begin{proof}  As $H_{n+1}$ is countable, it suffices to establish this claim for a single $g = g_{i_1} \ldots g_{i_r}$.  By approximation, we may also assume that $f$ is a simple function, that is to say a finite linear combination of indicator functions of internal subsets of $G$, and in particular is an internal function $f = \lim_{\n \to \alpha} f_\n$.  By the Fubini-Tonelli theorem (Theorem \ref{ftl}), we may replace the random variables $g_1,\ldots,g_n$ by deterministic elements of $G$, thus making $f$ deterministic, without affecting the joint distribution of $g_{n+1}, g_{n+2},\ldots$.

As in Lemma \ref{lrm}, $g$ is uniformly distributed on $G$.  The claim now follows from Lemma \ref{Drm-basic}.
\end{proof}

We will shortly use Lemma \ref{lrm} and Lemma \ref{Drm} to obtain a \emph{deterministic} IP system $H$ with good mixing properties.  To do this, we first need another technical lemma:

\begin{lemma}[Inclusion in a given set]\label{include}  Let $G, (g_a)_{a \in \N}, H$ be as above.  Let $E$ be a (deterministic) Loeb measurable subset of $G$.  Then for any $k \in \N$, the event 
\begin{equation}\label{event}
\{ g_{i_1} \ldots g_{i_r}: r \geq 1; 1 \leq i_1 < i_2 < \ldots < i_r \leq k \} \subset E 
\end{equation}
occurs with probability exactly $\mu(E)^{2^k-1}$.
\end{lemma}

\begin{proof}  We induct on $k$.  The case $k=0$ is trivial, so suppose $k \geq 1$ and the claim has already been proven for $k-1$.  We observe that the event \eqref{event} is the intersection of the events
\begin{equation}\label{event-1}
 g_k \in E
\end{equation}
and
\begin{equation}\label{event-2}
 \{ g_{i_1} \ldots g_{i_r}: r \geq 1; 1 \leq i_1 < i_2 < \ldots < i_r \leq k-1 \} \subset E \cap R_{g_k} E.
\end{equation}
The event \eqref{event-1} occurs with probability $\mu(E)$.  By Lemma \ref{wmi}, we almost surely\footnote{Given that our sample space is not complete, it may be worth clarifying that we say that a statement holds \emph{almost surely} if there is a measurable event in the sample space of probability $1$ for which the statement holds, allowing for the possibility that the statement might also hold on some (possibly non-measurable) subset of the complementary null set.} have 
\begin{equation}\label{rgk}
\mu(E \cap R_{g_k} E) = \mu(E)^2.
\end{equation}
If we replace the random variable $g_k$ by a deterministic element of $G$ that obeys \eqref{rgk}, then by the induction hypothesis, the event \eqref{event-2} would then occur with probability $(\mu(E)^2)^{2^{k-1}-1}$.  Applying the Fubini-Tonelli theorem (Theorem \ref{ftl}), we conclude that \eqref{event} occurs with probability $\mu(E) \times (\mu(E)^2)^{2^{k-1}-1} = \mu(E)^{2^k-1}$, as desired.
\end{proof}

\begin{remark} The same argument shows in fact that the random variables $g_{i_1} \ldots g_{i_r}$ for $r \geq 1$ and $1 \leq i_1 < \ldots < i_r$ are jointly independent and uniformly distributed in $G$, provided that one works with the ordinary product $\sigma$-algebra of all the copies of ${\mathcal B}_G$, rather than with the Loeb product $\sigma$-algebra constructed in Theorem \ref{prod}; we omit the details.  (The finitary version of this assertion is already implicit in the work of Gowers \cite{gowers}.)
\end{remark}

Now we can construct the deterministic IP system.

\begin{lemma}[Deterministic construction of a mixing IP system]\label{propo}  Let $G$ be an ultra quasirandom group, and let ${\mathcal X}_0$ be a (deterministic) separable sub-$\sigma$-algebra of ${\mathcal B}_G$.  Let $E$ be a Loeb measurable subset of $G$ of positive measure.  Then there exist a (deterministic) sequence $g_1,g_2,\ldots$ of elements of $G$ whose associated IP system
$$ H = \{ g_{i_1} \ldots g_{i_r}: r \geq 1; 1 \leq i_1 < i_2 < \ldots < i_r \}$$
obeys the the following properties, where ${\mathcal X}$ is the $\sigma$-algebra generated by the sets $L_g R_h E$ with $E \in {\mathcal X}_0$ and and $g,h \in \langle H\rangle$, and $\langle H\rangle$ is the group generated by $H$:
\begin{itemize}
\item[(i)] (Containment in $E$)  One has $H \subset E$.
\item[(ii)]  (Left and right mixing)  One has
$$ H\!-\!\lim_g \int_G f L_g f'\ d\mu_G = (\int_G f\ d\mu) (\int_X f'\ d\mu)$$
and
$$ H\!-\!\lim_g \int_G f R_g f'\ d\mu_G = (\int_G f\ d\mu) (\int_X f'\ d\mu)$$
for all $f,f' \in L^2(G, {\mathcal X}, \mu_G)$.
\item[(iii)] (Diagonal relative mixing)  One has
$$
H\!-\!\lim_g \| f - L_g R_g f \|_{L^2(G, {\mathcal B}_G, \mu_G)}^2
= 2 \| f - \E(f|{\mathcal I}_G) \|_{L^2(G, {\mathcal B}_G, \mu_G)}^2 
$$
for all $f \in L^2(G,{\mathcal X},\mu_G)$.
\end{itemize}
\end{lemma}

\begin{proof} If $E$ had full measure, the claim would be follow easily from Lemma \ref{include} (applied for $k=1,2,3,\ldots$), Lemma \ref{lrm}, and Lemma \ref{Drm}, since a random choice of $g_1,g_2,\ldots$ would then almost surely obey all the required properties.  Unfortunately, this argument does not work directly when $\mu_G(E)<1$, because the probability $\mu(E)^{2^k-1}$ appearing in Proposition \ref{include} then decays to zero as $k \to \infty$.  Nevertheless, one can still proceed in this case by using Proposition \ref{include}, Lemma \ref{lrm}, and Lemma \ref{Drm} to obtain a countable sequence of finitary truncations of Lemma \ref{propo}, and then appeal to countable compactness to then obtain the full strength of Lemma \ref{propo}.

We turn to the details.  Let ${\mathcal X}_0$ be generated by Loeb measurable sets $E_1,E_2,\ldots$.  By modifying each $E_i$ (and hence each set in ${\mathcal X}_0$) by a null set if necessary using Lemma \ref{appx}, we may assume without loss of generality that the $E_1,E_2,\ldots$ are \emph{internal} sets.  For each $k$, we may apply each of Proposition \ref{include}, Lemma \ref{lrm}, and Lemma \ref{Drm} a finite number of times to locate a finite deterministic sequence $(g_i)_{i=1}^k = (g_i^{(k)})_{i=1}^k$ of group elements obeying the following properties:

\begin{itemize}
\item[(i)] (Truncated containment in $E$)  One has
$$ \{ g_{i_1} \ldots g_{i_r}: r \geq 1; 1 \leq i_1 < i_2 < \ldots < i_r \leq k \} \subset E.$$
\item[(ii)]  (Truncated left and right mixing) For any $1 \leq k' < k$, any simple functions $f, f'$ that are linear combinations of at most $k$ indicator functions, each of which are boolean combinations of at most $k$ sets of the form $L_g R_h E_i$, where $i \leq k$, and $g,h$ are words of length at most $k$ in $g_1^{\pm},\ldots,g_{k'}^{\pm 1}$, with the coefficients of the linear combinations being rational with numerator and denominator bounded in magnitude by $k$, we have
$$ \int_G f L_g f'\ d\mu_G = (\int_G f\ d\mu) (\int_G f'\ d\mu)$$
and
$$ \int_G f R_g f'\ d\mu_G = (\int_G f\ d\mu) (\int_G f'\ d\mu)$$
whenever $g = g_{i_1} \ldots g_{i_r}$ for some $r \geq 1$ and $k' < i_1 < \ldots < i_r \leq k$.
\item[(iii)]  For any $1 \leq k' < k$, and $f$ and $g$ are as in (ii), we have
$$ \| f - L_g R_g f \|_{L^2(G, {\mathcal B}_G, \mu_G)}^2
= 2 \| f - \E(f|{\mathcal I}_G) \|_{L^2(G, {\mathcal B}_G, \mu_G)}^2.
$$
\end{itemize}

One can verify that for each $k$, the property of $(g_1,\ldots,g_k) \in G^k$ obeying the above properties describes a countably closed subset of $G^k$.  (Here we are implicitly using the continuous nature of the standard part function.)  On the other hand, from Lemma \ref{countsat} and Lemma \ref{sigma-tych} we know that the space $G^\N$ of sequences $(g_n)_{n \in \N}$ is countably compact.  We thus conclude the existence of an \emph{infinite} sequence $g_1,g_2,\ldots$ in $G$, such that the finite truncations $g_1,\ldots,g_k$ obey the above truncated properties for each $k$. 

By hypothesis, ${\mathcal X}_0$ can be generated by some countable sequence $E_1,E_2,\ldots$ of deterministic, Loeb measurable subsets of $G$.  Then ${\mathcal X}$ is generated by the countable family of random sets $L_g R_h E_n$, where $g,h \in \langle H \rangle$ lie in the group generated by $H$, and $n \in \N$.  Thus, any function in $f \in L^2(G,{\mathcal X},\mu_G)$ can be approximated to arbitrary accuracy (in $L^2$ norm) by a linear combination with rational coefficients by finitely many indicator functions, each of which is a boolean combination of sets of the form $L_g R_h E_i$, where $g, h$ are words in the $g_1,g_2,\ldots$.  The required claims (i)-(iii) of this lemma then follow from the truncated claims (i)-(iii) already established, and a limiting argument.
\end{proof}

As a consequence of this lemma and Theorem \ref{ergo}, we can obtain an ultraproduct version of Theorem \ref{main}:

\begin{theorem}[Relative weak mixing, ultraproduct version]\label{main-ultra}  Let $G$ be an ultra quasirandom group, and let $f_1,f_2,f_3 \in L^\infty(G, {\mathcal B}_G, \mu_G)$.  Then one has
$$ \int_G f_1 (L_g f_2) (L_g R_g f_3)\ d\mu_G = (\int_G f_2\ d\mu_G) (\int_G f_1 \E(f_3|{\mathcal I}_G) \ d\mu_G)$$
for $\mu_G$-almost all $g \in G$.
\end{theorem}

\begin{proof}  First suppose that $f_3$ is ${\mathcal I}_G$-measurable.  Then $L_g R_g f_3 = f_3$ for every $g$, and the claim then follows from Lemma \ref{lrm} (replacing $f_1$ by $f_1 f_3$).  Thus we may assume without loss of generality that $\E(f_3|{\mathcal I}_G)=0$, and the task is now to show that
$$ \int_G f_1 (L_g f_2) (L_g R_g f_3)\ d\mu_G = 0$$
for $\mu_G$-almost all $g \in G$.

Suppose this is not the case.  Then one can find a Loeb measurable set $E \subset G$ of positive Loeb measure, and an $\eps>0$, such that
\begin{equation}\label{gal}
 |\int_G f_1 (L_g f_2) (L_g R_g f_3)\ d\mu_G| > \eps 
\end{equation}
for all $g \in E$.  We then apply Lemma \ref{propo}, with ${\mathcal X}_0$ equal to the (separable) $\sigma$-algebra generated by $f_1,f_2,f_3$, to find an IP system $H$ inside $E$ obeying all the conclusions of that proposition.  

Let $(X,{\mathcal X},\mu)$ be the restriction of $(G,{\mathcal B}_G, \mu_G)$ to the $\sigma$-algebra ${\mathcal X}$ generated by ${\mathcal X}_0$ and the shifts $L_g, R_g$ for $g \in H$.  We now verify the hypotheses of Theorem \ref{ergo}.  From Lemma \ref{propo}(ii) we have
$$ H\!-\!\lim_g \int_X f L_g f'\ d\mu = (\int_X f\ d\mu) (\int_X f'\ d\mu)$$
and
$$ H\!-\!\lim_g \int_X f R_g f'\ d\mu = (\int_X f\ d\mu) (\int_X f'\ d\mu).$$
for all $f,f' \in L^2(X,{\mathcal X},\mu)$, which are the hypotheses (i) and (ii) for Theorem \ref{ergo}.

Now we verify hypothesis (iii) for Theorem \ref{ergo}.  Suppose that $f \in L^2(X,{\mathcal X},\mu)$ be any function with the rigidity property that, for any $\eps>0$ and natural number $n$, there exists $g \in H_n$ such that $\| L_g R_g f - f \|_{L^2(X,{\mathcal X},\mu)} \leq \eps$.  
By Proposition \ref{propo}(iii), we conclude that
$$\| f - \E(f|{\mathcal I}_G) \|_{L^2(G, {\mathcal B}_G, \mu_G)} = 0,$$
thus $f$ is ${\mathcal I}_G$ measurable up to $\mu_G$-almost everywhere equivalence.  Since $\E(f_3|{\mathcal I}_G)=0$, we conclude that $\int_X f f_3\ d\mu = 0$, giving hypothesis (iii) for Theorem \ref{ergo}.  We may then apply Theorem \ref{ergo} to find $g \in H$ such that
$$ |\int_X f_1 (L_g f_2) (L_g R_g f_3)\ d\mu| \leq \eps.$$
But this contradicts \eqref{gal} and Lemma \ref{propo}(i), and the claim follows.
\end{proof}

\begin{remark}  One can reformulate the conclusion of Theorem \ref{main-ultra} as the assertion that the pushforward of the Loeb measure $\mu_{G^2}$ to $G^4$ under the map $(x,g) \mapsto (g,x,xg,gx)$, when restricted to the product $\sigma$-algebra ${\mathcal B}_G \times {\mathcal B}_G \times {\mathcal B}_G \times {\mathcal B}_G$, is equal to $\mu_G \times \mu_G \times (\mu_G \times_{{\mathcal I}_G} \mu_G)$, where $\mu_G \times_{{\mathcal I}_G} \mu_G$ is the relative product of the measures $\mu_G$ with respect to the factor ${\mathcal I}_G$.
\end{remark}

Finally, we can use the ultraproduct correspondence principle to recover Theorem \ref{main}.

\begin{proof}[Proof of Theorem \ref{main}]  A simple change of variables reveals the identity
$$ \E_{x \in G} f_1(x) f_2(xg) f_3(gx) = \int_G f_2 (L_g f_1) (L_g R_g f_3)\ d\mu_G$$
for any finite group $G$ and functions $f_1,f_2,f_3: G \to \R$.  Thus, it suffices to show that for any $D$-quasirandom group $G$, and any functions $f_1,f_2,f_3: G \to [-1,1]$, one has
$$ \int_G |\int_G f_2 (L_g f_1) (L_g R_g f_3)\ d\mu_G - (\int_G f_1\ d\mu_G) (\int_G f_2 \E(f_3|{\mathcal I}_G)\ d\mu_G)|\ d\mu_G(g) \leq c(D)$$
for some $c(D)$ going to zero as $D \to \infty$.

Suppose for sake of contradiction that this claim failed.  Carefully negating the quantifiers, we may then find an $\eps>0$ and a sequence $G_\n$ of finite groups and functions $f_{1,\n}, f_{2,\n}, f_{3,\n}: G_\n \to [-1,1]$, such that for each $\n$, $G_\n$ is $\n$-quasirandom and
\begin{equation}\label{lode}
 \int_{G_\n} |\int_{G_\n} f_{2,\n} (L_{\n,g_\n} f_{1,\n}) (L_{\n,g_\n} R_{\n,g_\n} f_{3_\n})\ d\mu_{G_\n} - (\int_{G_\n} f_{1_\n}\ d\mu_{G_\n}) (\int_{G_\n} f_{2_\n} \E(f_{3_\n}|{\mathcal I}_{G_\n})\ d\mu_{G_\n})|\ d\mu_{G_\n}(g_\n) \geq \eps.
\end{equation}
If we now form the ultraproduct $G := \prod_{\n \to \alpha} G_\n$ and the functions $f_i := \st \lim_{\n \to \alpha} f_{i,\n}$, then $G$ is an ultra quasirandom group and $f_1,f_2,f_3 \in L^\infty(G, {\mathcal B}_G, \mu_G)$.  Taking ultralimits of \eqref{lode} (and using Lemma \ref{conj}) we see that
$$
 \int_{G} |\int_G f_2 (L_g f_1) (L_g R_g f_3)\ d\mu_G - (\int_{G} f_{1}\ d\mu_{G}) (\int_{G} f_{2} \E(f_{3}|{\mathcal I}_{G})\ d\mu_{G})|\ d\mu_{G}(g) \geq \eps.$$
But this contradicts Theorem \ref{main-ultra}, and Theorem \ref{main} follows.
\end{proof}

\section{Proof of Theorem \ref{loaded-2}}\label{lo2-sec}

We now give two proofs of Theorem \ref{loaded-2}: a combinatorial proof, and a proof using the ultraproduct correspondence principle.  In both proofs, we use the fact that in a finite group $G$ with a subset $A \subset G$, the number of pairs $(x,g) \in G^2$ with $x, xg, gx \in A$ is also equal to
\begin{equation}\label{form}
 |G|^2 \int_{G^3} 1_A(ab) 1_A( aca^{-1} ) 1_A( bcb^{-1} ) d\mu_{G^3}(a,b,c),
 \end{equation}
as can be seen by applying the $|G|$-to-one change of variables $(x,g) := (ab, bc^{-1} a)$.  Note that each of the three factors $1_A(ab)$, $1_A( aca^{-1} )$, $1_A( bcb^{-1} )$ depends on only two of the three variables $a,b,c$.

We begin with the combinatorial proof.  The main tool is the triangle removal lemma of Ruzsa and Szemer\'edi:

\begin{lemma}[Triangle removal lemma]\label{tri}  For every $\delta>0$ there exists $\eps>0$ such that if $G = (V,E)$ is a graph on $n$ vertices with at most $\eps n^3$ triangles, then it is possible to remove fewer than $\delta n^2$ edges from the graph to obtain a graph with no triangles whatsoever.
\end{lemma}

\begin{proof} See \cite{rsz}.  The main ingredient of the proof is the Szemer\'edi regularity lemma \cite{szemeredi-reg}.
\end{proof}

Now we prove Theorem \ref{loaded-2}.  Let $\delta>0$, and let $\eps>0$ be sufficiently small depending on $\delta$.  Suppose for contradiction that we can find a finite group $G$ and a subset $A$ of $G$ with $|A| \geq \delta |G|$, and such that there are at most $\eps |G|^2$ pairs $(x,g) \in G \times G$ with $x,gx,xg\in G$; using the formula \eqref{form}, we conclude that
\begin{equation}\label{abc}
 \E_{a,b,c \in G} 1_A(ab) 1_A( aca^{-1} ) 1_A( bcb^{-1} ) \leq \eps.
\end{equation}
Now consider the tripartite graph $(V,E)$ with $V = G \times \{1,2,3\}$, and $E$ give by the following edges:
\begin{itemize}
\item If $(a,1), (b,2) \in V$ are such that $ab \in A$, then $\{(a,1),(b,2)\} \in E$.
\item If $(a,1), (c,3) \in V$ are such that $aca^{-1} \in A$, then $\{(a,1),(c,3)\} \in E$.
\item If $(b,2), (c,3) \in V$ are such that $bcb^{-1} \in A$, then $\{(b,2),(c,3)\} \in E$.
\item There are no further edges.
\end{itemize}
From \eqref{abc} we see that $(V,E)$ contains $3|G|$ vertices and at most $\eps |G|^3$ triangles.  Applying Lemma \ref{tri}, we see (for $\eps$ small enough) that one can remove all the triangles from $(V,E)$ by deleting fewer than $\delta |G|^2$ edges.  On the other hand, each pair $(a,b) \in G$ with $ab \in A$ leads to a triangle in $(V,E)$ with vertices $(a,1), (b,2), (ba,3)$.  There are at least $|A| |G| \geq \delta|G|^2$ such triangles, and the edges in these triangles are all disjoint, and so at least $\delta |G|^2$ edges need to be deleted in order to remove all triangles.  This gives the desired contradiction, and Theorem \ref{loaded-2} follows. $\Box$

\begin{remark}  The above argument in fact gives a quantitative value for $\eps$ which is of tower-exponential type with respect to $\delta$.  It would be of interest to obtain any improvement to this bound.
\end{remark}

Now we give the ultraproduct proof.  We can deduce Theorem \ref{loaded-2} from its ultraproduct version:

\begin{theorem}[Strong recurrence, ultraproduct version]\label{loaded-2-ultra}  Let $G$ be an ultraproduct of finite groups, and let $A$ be a Loeb measurable subset of $G$ with $\mu_G(A) > 0$.  Then 
$$ \int_{G^3} 1_A(ab) 1_A(aca^{-1}) 1_A(bcb^{-1}) d\mu_{G^3}(a,b,c) > 0.$$
\end{theorem}

The derivation of Theorem \ref{loaded-2} from Theorem \ref{loaded-2-ultra} (using \eqref{form}) is a routine (and simpler) variant of the derivation of Theorem \ref{loaded} from Proposition \ref{p1}, or Theorem \ref{main} from Theorem \ref{main-ultra}, and is omitted.  

We will need the following ultraproduct variant of the triangle removal lemma:

\begin{lemma}[Ultraproduct triangle removal lemma]\label{ulrt}  Let $V_1,V_2,V_3$ be the ultraproducts of finite non-empty sets, and let $A_{12}, A_{23}, A_{13}$ be Loeb measurable subsets of $V_1 \times V_2$, $V_2 \times V_3$, $V_3 \times V_1$ respectively.  Suppose that
$$ \int_{V_1 \times V_2 \times V_3} 1_{A_{12}}(a,b) 1_{A_{23}}(b,c) 1_{A_{13}}(a,c)\ d\mu_{G^3}(a,b,c) = 0.$$
Then for any $\eps>0$, there exist Loeb measurable subsets $A'_{12}, A'_{23}, A'_{13}$ of $V_1 \times V_2$, $V_2 \times V_3$, $V_3 \times V_1$ respectively respectively with
\begin{equation}\label{ooa}
 1_{A'_{12}}(a,b) 1_{A'_{23}}(b,c) 1_{A'_{13}}(a,c) = 0
\end{equation}
for all $a \in V_1, b \in V_2, c \in V_3$, and
\begin{equation}\label{mau}
 \mu_{V_i \times V_j}(A_{ij} \Delta A'_{ij}) \leq \eps
 \end{equation}
for $ij=12,23,13$.
\end{lemma}

This lemma was proven in \cite{EleSze} and \cite{Tao3}; for the sake of completeness, we give a proof later in this section.  Assuming this lemma for now, let us conclude the proof of Theorem \ref{loaded-2-ultra}.   Suppose for contradiction that
\begin{equation}\label{g3}
\int_{G^3} 1_A(ab) 1_A(aca^{-1}) 1_A(bcb^{-1}) d\mu_{G^3}(a,b,c) = 0.
\end{equation}
Let $\eps>0$ be chosen later.  Applying Lemma \ref{ulrt}, we can find Loeb measurable subsets $A'_{12}, A'_{23}, A'_{13}$ of $G^2$ such that
\begin{align}
 \mu_{G^2}( \{ (a,b) \in G^2: ab \in A; (a,b) \not \in A'_{12} \} ) &\leq \eps \label{lo}\\
 \mu_{G^2}( \{ (a,c) \in G^2: aca^{-1} \in A; (a,c) \not \in A'_{13} \} ) &\leq \eps\label{lo-1}\\
 \mu_{G^2}( \{ (b,c) \in G^2: bcb^{-1} \in A; (b,c) \not \in A'_{23} \} ) &\leq \eps\label{lo-2}.
 \end{align}
and
$$ 1_{A'_{12}}(a,b) 1_{A'_{13}}(a,c) 1_{A'_{23}}(b,c) = 0$$
for all $a,b,c \in G$.  In particular, one has
\begin{equation}\label{gaa}
 \int_{G^2}  1_{A'_{12}}(a,b) 1_{A'_{13}}(a,ba) 1_{A'_{23}}(b,ba)\ d\mu_{G^2}(a,b) = 0.
\end{equation}
Using \eqref{lo-1}, \eqref{lo-2} and the change of variables $c=ba$, we see that
\begin{align*}
 \mu_{G^2}( \{ (a,b) \in G^2: ab; (a,ba) \not \in A'_{13} ) &\leq \eps \\
  \mu_{G^2}( \{ (a,b) \in G^2: ab \in A; (b,ba) \not \in A'_{23} ) &\leq \eps
 \end{align*}
and from these bounds, \eqref{lo}, and \eqref{gaa} we conclude that
$$ \int_{G^2} 1_A(ab)\ d\mu_{G^2}(a,b) \leq 3\eps$$
and thus $\mu(A) \leq 3\eps$.  Since $\mu(A)$ was assumed to be positive, we obtain a contradiction for $\eps$ small enough, establishing Theorem \ref{loaded-2-ultra} and hence Theorem \ref{loaded-2}.

Now we prove Lemma \ref{ulrt}; this will be the standard proof of Lemma \ref{tri}, converted into ultraproduct form.  Observe that for each $c \in V_3$, the function $(a,b) \mapsto 1_{A_{13}}(a,c) 1_{A_{23}}(b,c)$ is measurable with respect to the product $\sigma$-algebra ${\mathcal B}_{V_1} \times {\mathcal B}_{V_2}$.  Thus we have
\begin{align*}
& \int_{V_1 \times V_2} 1_{A_{12}}(a,b) 1_{A_{13}}(a,c) 1_{A_{23}}(b,c) d\mu_{V_1 \times V_2}(a,b) \\
&\quad = 
\int_{V_1 \times V_2} \E(1_{A_{12}}|{\mathcal B}_{V_1} \times {\mathcal B}_{V_2})(a,b) 1_{A_{13}}(a,c) 1_{A_{23}}(b,c) d\mu_{V_1 \times V_2}(a,b).
\end{align*}
Integrating in $c$ using the Fubini-Tonelli theorem (Theorem \ref{ftl}) we conclude that
$$
\int_{V_1 \times V_2 \times V_3} \E(1_{A_{12}}|{\mathcal B}_{V_1} \times {\mathcal B}_{V_2})(a,b) 1_{A_{13}}(a,c) 1_{A_{23}}(b,c) d\mu_{V_1 \times V_2 \times V_3}(a,b,c) = 0.$$
Arguing similarly using the other two factors, we conclude that
$$ \int_{V_1 \times V_2 \times V_3} f_{12}(a,b) f_{13}(a,c) f_{23}(b,c)\ d\mu_{V_1 \times V_2 \times V_3}(a,b,c) = 0$$
where $f_{ij} := \E(1_{A_{ij}}|{\mathcal B}_{V_i} \times {\mathcal B}_{V_j} )$ for $ij=12,13,23$.

For each $ij=12,13,23$, let $\tilde A_{ij} \subset G^2$ be the set $\tilde A_{ij} := \{ x \in V_i \times V_j: f_{ij} \geq \eps/4 \}$.  Since we have the pointwise bound
$$ 0 \leq 1_{\tilde A_{ij}} \leq \eps^{-1} f_i$$
we conclude that
\begin{equation}\label{loe}
 \int_{V_1 \times V_2 \times V_3} 1_{\tilde A_{1,2}}(a,b) 1_{\tilde A_{13}}(a,c) 1_{\tilde A_{23}}(b,c)\ d\mu_{V_1 \times V_2 \times V_3}(a,b,c) = 0.
\end{equation}
For each $ij=12,13,23$, the sets $A_{ij}$ are measurable with respect to the product topology ${\mathcal B}_{V_i} \times {\mathcal B}_{V_j}$, which is generated by product sets $E_i \times F_j$ for $E_i \in {\mathcal B}_{V_i}$, $F_j \in {\mathcal B}_{V_j}$.  Approximating $A_{ij}$ to error $\eps/4$ by a finite combination of these sets, we can find finite sub-$\sigma$-algebras ${\mathcal B}'_{V_i,ij}, {\mathcal B}'_{V_j,ij}$ of ${\mathcal B}_{V_i}, {\mathcal B}_{V_j}$ respectively  such that
$$ \| 1_{\tilde A_{ij}} - \E(1_{\tilde A_{ij}}|{\mathcal B}'_{V_i,ij} \times {\mathcal B}'_{V_j,ij}) \|_{L^1(V_i \times V_j, {\mathcal B}_{V_i} \times {\mathcal B}_{V_j}, \mu_{V_i \times V_j})} \leq \eps/4.$$
By combining the finite factors together, we thus obtain a single finite factor ${\mathcal B}'_{V_i}$ for each $i=1,2,3$ with the property that
\begin{equation}\label{sor}
\| 1_{\tilde A_{ij}} - \E(1_{\tilde A_{ij}}|{\mathcal B}'_{V_i} \times {\mathcal B}'_{V_j}) \|_{L^1(V_i \times V_j, {\mathcal B}_{V_i} \times {\mathcal B}_{V_j}, \mu_{V_i\times V_j})} \leq \eps/4
\end{equation}
for $ij=12,13,23$.   By absorbing atoms of zero measure, we can assume that all atoms in the ${\mathcal B}'_{V_i}$ have positive measure.

For $ij=12,13,23$, let $A'_{ij}$ be the restriction of $A_{ij}$ to those atoms $E_i \times F_j$ of ${\mathcal B}'_{V_i} \times {\mathcal B}'_{V_j}$ for which
\begin{equation}\label{lemon}
 \E(1_{\tilde A_{ij}}|{\mathcal B}'_{V_i} \times {\mathcal B}'_{V_j}) > 2/3.
\end{equation}
We claim that \eqref{ooa} holds for any $a \in V_1, b \in V_2, c \in V_3$.  Indeed, let $E_1,E_2,E_3$ be the atoms of ${\mathcal B}'_{V_1},{\mathcal B}'_{V_2},{\mathcal B}'_{V_3}$ containing $a,b,c$ respectively.  From \eqref{lemon} and the Fubini-Tonelli theorem, we see that the sets
\begin{align*}
\{ (a,b,c) \in E_1 \times E_2 \times E_3: &(a,b) \in \tilde A_{12} \} \\
\{ (a,b,c) \in E_1 \times E_2 \times E_3: &(a,c) \in \tilde A_{13} \} \\
\{ (a,b,c) \in E_1 \times E_2 \times E_3: &(b,c) \in \tilde A_{23} \} 
\end{align*} 
each have density greater than $2/3$ in $E_1 \times E_2 \times E_3$, and hence the set
$$ \{ (a,b,c) \in E_1 \times E_2 \times E_3: (a,b) \in \tilde A_{12}; (a,c) \in \tilde A_{13}, \tilde A_{23} \} $$
has positive measure, contradicting \eqref{loe}.  This establishes \eqref{ooa}.

Finally, we need to show \eqref{mau}.  For sake of notation, we show this for $ij=12$, as the other two cases are analogous.
By construction, the set $A_{12} \Delta A'_{12}$ is contained in the set
$$ \{ (a,b) \in A_{12}: f_{12} < \eps/4 \} \cup \{ (a,b) \in \tilde A_{12}: \E(1_{\tilde A_{12}}|{\mathcal B}'_{V_1} \times {\mathcal B}'_{V_2}) > 2/3 \}$$
and so
$$\mu_{V_1 \times V_2}(A_{12} \Delta A'_{12}) \leq \int_{V_1 \times V_2} 1_{A_{12}} 1_{f_{12} < \eps/4}
+ 1_{\tilde A_{12}} 1_{\E(1_{\tilde A_{12}}|{\mathcal B}'_{V_1} \times {\mathcal B}'_{V_2}) > 2/3}\ d\mu_{V_1 \times V_2}.$$
The right-hand side can be written as the sum of
$$ \int_{V_1 \times V_2} f_{12} 1_{f_{12} < \eps/4}\ d\mu_{V_1 \times V_2}$$
and
$$  \int_{V_1 \times V_2} 1_{\tilde A_{12}} 1_{1_{\tilde A_{12}} - \E(1_{\tilde A_{12}}|{\mathcal B}'_{V_1} \times {\mathcal B}'_{V_2}) < 1/3}\ d\mu_{V_1 \times V_2}.
$$
The first integral is at most $\eps/4$, while the second expression is at most
$$ 3 \| 1_{\tilde A_{12}} - \E(1_{\tilde A_{12}}|{\mathcal B}'_{V_1} \times {\mathcal B}'_{V_2})\|_{{L^1(V_i \times V_j, {\mathcal B}_{V_i} \times {\mathcal B}_{V_j}, \mu_{V_i\times V_j})}}$$
which by \eqref{sor} is at most $3\eps/4$.  The claim \eqref{mau} follows.

Now we can prove Theorem \ref{loaded-4}.  We first observe that this theorem follows from an apparently weaker version in which the final condition
$ (gx_1g^{-1},\ldots,gx_kg^{-1}) \in A$ is deleted.  Namely, we will deduce Theorem \ref{loaded-4} from the following result:

\begin{theorem}[Multiple strong recurrence]\label{loaded-3}  Let $k \geq 1$ be a natural number.  For every $\delta>0$, there exists $\eps>0$ such that the following statement holds: if $G$ is a finite group, and $A$ is a subset of $G^k$ with $|A| \geq \delta |G|^k$, then there exist at least $\eps |G|^{k+1}$ tuples $(g,x_1,\ldots,x_k) \in G^{k+1}$ such that\footnote{Recall that by the convention of ignoring the initial block $gx_1,\ldots,gx_i$ when $i=0$ and the final block $x_{i+1},\ldots,x_k$ when $i=k$, we interpret $(gx_1,\ldots,gx_i,x_{i+1},\ldots,x_k)$ as $(x_1,\ldots,x_k)$ when $i=0$ and $(gx_1,\ldots,gx_k)$ when $x=k$.} $(gx_1,\ldots,gx_i,x_{i+1},\ldots,x_k) \in A$ for all $i=0,\ldots,k$.
\end{theorem}

Indeed, if $k, \delta, \eps, A$ are as in Theorem \ref{loaded-4}, one can conclude that theorem by applying Theorem \ref{loaded-3} with $k$ replaced by $k+1$ and $A$ replaced by the set
$$ \{ (x_1,\ldots,x_{k+1}) \in G: (x_1 x_{k+1}^{-1}, \ldots, x_k x_{k+1}^{-1}) \in A \};$$
we leave the routine verification of this implication to the reader.

We also remark that a variant of Theorem \ref{loaded-3} can be proven by the arguments used to establish \cite[Corollary 6.4]{berg-amenable}, as noted in the comments after that corollary.  In this variant, the conclusion is instead that there exist at least $\eps |G|^{k+1}$ tuples $(g,x_1,\ldots,x_k) \in G^{k+1}$ such that $(x_1,\ldots,x_k) \in A$ and  $(x_1,\ldots,x_{i-1},gx_i,x_{i+1},\ldots,x_k) \in A$ for all $i=1,\ldots,k$.  

We now prove Theorem \ref{loaded-3}.  We will generalize the combinatorial proof of Theorem \ref{loaded-2}, by replacing the triangle removal lemma of Ruzsa and Szemer\'edi with the more general \emph{hypergraph removal lemma} first established in \cite{rodl}, \cite{rs}, \cite{gowers-hyper}.  (The measure-theoretic proof also generalizes, but we leave this as an exercise to the interested reader.) We will use the following special case of this lemma:

\begin{lemma}[Simplex removal lemma]\label{simplicio}  
Let $k \geq 1$ be an integer.  For every $\delta>0$ there exists $\eps>0$ such that if $V_0,\ldots,V_k$ are sets of $n$ vertices, and for each $i=0,\ldots,k$, $E_i \subset V_0 \times \ldots \times V_{i-1} \times V_{i+1} \times \ldots \times V_k$ is a set with the property that\footnote{Continuing the previous block-ignoring convention, we interpret $( x_0,\ldots,x_{i-1},x_{i+1},\ldots,x_k)$ as $(x_1,\ldots,x_k)$ when $i=0$ and $(x_0,\ldots,x_{k-1})$ when $i=k$.}
\begin{equation}\label{deltank}
 \sum_{x_0 \in V_0, \ldots, x_k \in V_k} \prod_{i=0}^k 1_{E_i}( x_0,\ldots,x_{i-1},x_{i+1},\ldots,x_k) \leq \eps n^{k+1},
\end{equation}
then it is possible to remove fewer than $\delta n^k$ elements from $E_i$ for each $i=0,\ldots,k$ to form a subset $E'_i$ such that
$$ \sum_{x_0 \in V_0, \ldots, x_k \in V_k} \prod_{i=0}^k 1_{E'_i}( x_0,\ldots,x_{i-1},x_{i+1},\ldots,x_k) = 0.$$
\end{lemma}

\begin{proof}  This is a special case of \cite[Theorem 1.13]{tao-hyper}.
\end{proof}

Let $k,\delta$ be as in Theorem \ref{loaded-3}, let $\eps>0$ be a sufficiently small quantity, and let $G, A$ obey the hypotheses of Theorem \ref{loaded-3}.  Suppose for contradiction that there are fewer than $\eps |G|^{k+1}$ tuples $(g,x_1,\ldots,x_k) \in G^{k+1}$ such that $(gx_1,\ldots,gx_i,x_{i+1},\ldots,x_k) \in A$ for all $i=0,\ldots,k$.

For each $i=0,\ldots,k$, we set $V_i := G$, and then let $E_i \subset G^k$ be the set of all tuples $(x_0,\ldots,x_{i-1},x_{i+1},\ldots,x_k) \in G^k$ with the property that the $k$-tuple\footnote{Continuing previous conventions, we ignore the block $x_0, x_0 x_1, \ldots, x_0 \ldots x_{i-1}$ when $i=0$, and $x_k^{-1} \ldots x_{i+1}^{-1}, \ldots, x_k^{-1} x_{k-1}^{-1}, x_k^{-1}$ when $i=k$.}
$$ (x_0, x_0 x_1, \ldots, x_0 \ldots x_{i-1}, x_k^{-1} \ldots x_{i+1}^{-1}, \ldots, x_k^{-1} x_{k-1}^{-1}, x_k^{-1})$$
lies in $A$.  
For instance, if $k=3$, we have
\begin{align*}
E_0 &= \{ (x_1,x_2,x_3) \in G^3: (x_3^{-1}x_2^{-1}x_1^{-1}, x_3^{-1}x_2^{-1}, x_3^{-1}) \in A \} \\
E_1 &= \{ (x_0,x_2,x_3) \in G^3: (x_0, x_3^{-1}x_2^{-1}, x_3^{-1}) \in A \} \\
E_2 &= \{ (x_0,x_1,x_3) \in G^3: (x_0, x_0 x_1, x_3^{-1}) \in A \} \\
E_3 &= \{ (x_0,x_1,x_2) \in G^3: (x_0, x_0 x_1, x_0 x_1 x_2) \in A \}.
\end{align*}
Now suppose that $(x_0,\ldots,x_k)$ makes a non-zero contribution to the left-hand side of \eqref{deltank}, thus
$$ (x_0, x_0 x_1, \ldots, x_0 \ldots x_{i-1}, x_k^{-1} \ldots x_{i+1}^{-1}, \ldots, x_k^{-1} x_{k-1}^{-1}, x_k^{-1}) \in A$$
for all $i=0,\ldots,k$.  If we then define
$$ y_i := x_k^{-1} \ldots x_i^{-1}$$
for $i=1,\ldots,k$, and
$$ g := x_0 \ldots x_k,$$
we conclude that
$$ (gy_1,\ldots,gy_i,y_{i+1},\ldots,y_k) \in A$$
for $i=0,\ldots,k$.  From our hypotheses, we conclude that \eqref{deltank} holds.  Applying Lemma \ref{simplicio} (with $\delta$ replaced by $\delta/(k+1)$), we conclude (for $\eps$ small enough) that we can remove fewer than $\frac{\delta}{k+1} |G|$ elements from $E_i$ to create a subset $E'_i$, with the property that there do not exist any tuples $(x_0,\ldots,x_{k+1}) \in G^{k+1}$ with the property that $(x_0,\ldots,x_{i-1},x_{i+1},\ldots,x_k) \in E'_i$ for all $0 \leq i \leq k$.

Let $(y_1,\ldots,y_k)$ be an element of $A$.  Applying the previous claim with
$$ (x_0,\ldots,x_k) := (y_0^{-1} y_1, y_1^{-1} y_2, \ldots, y_k^{-1} y_{k+1})$$
with the convention that $y_0 = y_{k+1}=1$, we see that there is at least one $0 \leq i \leq k$ such that
$$ (y_0^{-1} y_1, \ldots, y_{i-2}^{-1} y_{i-1}, y_i^{-1} y_{i+1}, \ldots, y_k^{-1} y_{k+1}) \not \in E'_i$$
(using the same block-ignoring conventions as before).  On the other hand, from definition of $E_i$ and the hypothesis $(y_1,\ldots,y_k) \in A$, we see that
$$ (y_0^{-1} y_1, \ldots, y_{i-2}^{-1} y_{i-1}, y_i^{-1} y_{i+1}, \ldots, y_k^{-1} y_{k+1}) \in E_i.$$
Applying the pigeonhole principle, we conclude that there exist $0 \leq i \leq k$ such that
$$ (y_0^{-1} y_1, \ldots, y_{i-2}^{-1} y_{i-1}, y_i^{-1} y_{i+1}, \ldots, y_k^{-1} y_{k+1}) \in E_i \backslash E'_i$$
for at least $|A|/(k+1) \geq \frac{\delta}{k+1} |G|$ tuples $(y_1,\ldots,y_k)$, thus $|E_i \backslash E'_i| \geq \frac{\delta}{k+1} |G|$.  But this contradicts the construction of $E'_i$, and Theorem \ref{loaded-3} follows. $\Box$

\section{Remarks on specific ultra quasirandom groups}

In this section, $\alpha \in \beta \N \backslash \N$ is a fixed non-principal ultrafilter.

In Section \ref{ultrasec}, some general mixing properties were obtained for arbitrary ultra quasirandom groups.  It turns out that for some specific examples of ultra quasirandom groups, one can obtain further mixing properties, particularly for ultraproducts of the finite groups $SL_2(F_p)$, the mixing properties of which have been intensively studied.  Indeed, thanks to the existing literature on such groups, we have the following results:

\begin{theorem}[Mixing properties of $SL_2(F_p)$]\label{mix}  Let $p_\n$ be a sequence of primes going to infinity, let $F$ be the characteristic zero pseudo-finite field\footnote{In model theory, a \emph{pseudo-finite field} is a field which obeys all first-order sentences in the language of fields that are true in all finite fields (for instance, a pseudo-finite field has exactly one field extension of each finite degree).  In particular, any ultraproduct of finite fields is a pseudo-finite field.} $F := \prod_{\n \to \alpha} F_{p_\n}$, and let $G$ be the ultra quasirandom group $G:= SL_2(F) = \prod_{\n \to \alpha} SL_2(F_{p_\n})$. 
\begin{itemize}
\item[(i)]  (Weak mixing) $G$ has no non-trivial finite-dimensional unitary representations; thus, for any $d \geq 1$, the only homomorphism from $G$ to $U_d(\C)$ is the trivial one.
\item[(ii)]  (Almost sure expansion) There is an absolute constant $\eps>0$ with the property that for $\mu_{G^2}$-almost every pair $(a,b) \in G$, one has the spectral gap property
\begin{equation}\label{loo}
 \| \frac{1}{4} (L_a + L_b + L_{a^{-1}} + L_{b^{-1}}) \|_{\operatorname{op}} \leq 1 - \eps,
\end{equation}
where $\| \|_{\operatorname{op}}$ denotes the operator norm on the space $L^2(G, {\mathcal B}_G, \mu_G)_0$ of mean zero functions in $L^2(G, {\mathcal B}_G, \mu_G)$.
\item[(iii)] (Uniform expansion in most cases)  There exists a universal subset $A$ of the primes of density zero, such that if the primes $p_\n$ all avoid this set, then the spectral gap property \eqref{loo} holds for \emph{all} pairs $(a,b) \in G$ which generate a Zariski-dense subgroup of $G$.
\item[(iv)]  (Uniform expansion in a $SL_2(\Z)$ component)  Identifying $\Z$ with the subring generated by the identity $1$ of $F$, the spectral gap property \eqref{loo} holds for any $(a,b) \in SL_2(\Z)$ generating a Zariski-dense subgroup of $SL_2(\Z)$.
\item[(v)] (Lack of mild mixing\footnote{A measure-preserving system is said to be \emph{mild mixing} if there are no non-trivial rigid functions; in the case of actions of abelian groups, this concept is intermediate in strength between weak mixing and strong mixing.})  For any $a \in G$, there exists a Loeb-measurable subset $E$ of $G$ such that $\mu_G(E)=1/2$ and $L_a E = E$.  (In particular, $\mu_G( L_{a^n} E \cap E ) \not \to \mu_G(E)^2$ as $n \to \infty$.
\end{itemize}
\end{theorem}

Of course, similar results hold if the left shift $L_g$ is replaced with the right shift $R_g$ throughout.  Among other things, the uniform expansion properties of the ultra quasirandom group $G$ established in the above theorem suggest that this group behaves very ``non-amenably''.  In general, it appears that ergodic theory tools that are restricted to amenable group actions are not suitable for the analysis of ultra quasirandom groups.

\begin{proof}  We begin with (i).  This does not seem to follow directly from Lemma \ref{frob}, but can be deduced from modifying the \emph{proof} of that lemma, as follows.   Suppose for contradiction that we have a non-trivial representation $\rho: SL_2(F) \to U_d(\C)$ on a unitary group of some finite dimension $d$.  Set $a$ to be the group element
$$ a := \begin{pmatrix} 1 & 1 \\ 0 & 1 \end{pmatrix},$$
and suppose first that $\rho(a)$ is non-trivial.  Arguing as in the proof of Lemma \ref{frob}, we see that the eigenvalues of $\rho(a)$ are permuted by the operation $x \mapsto x^m$ for any perfect square $m \in \N$, because $a$ is conjugate to $a^{m}$.  In particular, this implies that all the eigenvalues are roots of unity; clearing denominators, we see that $\rho(a^m)=1$ for some perfect square $m \in \N$, and hence $\rho(a)=1$.  Conjugating again, this time by the diagonal matrix with entries $m,m^{-1}$ for a non-zero $F$, we see that
$$ \rho( \begin{pmatrix} 1 & m^2 \\ 0 & 1 \end{pmatrix} ) = 1$$
for all $m \in F$.  In each finite field $F_{p_\n}$, it is a classical fact that every residue class is the sum of three quadratic residues; taking ultraproducts, the same claim is true in $F$.  As $\rho$ is a homomorphism, we thus see that
$$ \rho( \begin{pmatrix} 1 & t \\ 0 & 1 \end{pmatrix} ) = 1$$
for any $t \in F$.  By conjugation, we thus also have
$$ \rho( \begin{pmatrix} 1 & 0 \\ t & 1 \end{pmatrix} ) = 1.$$
These two one-parameter groups of matrices are easily verified to generate $SL_2(F)$, and so $\rho$ is trivial, giving the desired contradiction.

Now we establish (ii).  We will use (as a black box) one of the main theorems \cite[Theorem 2]{bourgain-gamburd} of Bourgain and Gamburd, which in our notation asserts that that for each prime $p_\n$ there exists an exceptional set $E_\n$ of $SL_2(F_{p_\n}) \times SL_2(F_{p_\n})$ of density $\mu_{SL_2(F_{p_\n}) \times SL_2(F_{p_\n})}(E_\n)$ going to zero as $\n \to \infty$, and an absolute constant $\eps>0$, such that
\begin{equation}\label{alo}
 \| \frac{1}{4} (L_{a_\n} + L_{b_\n} + L_{a^{-1}_\n} + L_{b^{-1}_\n}) \|_{\operatorname{op}} \leq 1 - \eps,
\end{equation}
for all $(a_\n, b_\n) \in SL_2(F_{p_\n}) \times SL_2(F_{p_\n}) \backslash E_\n$.  If we let $E := \prod_{\n \to \alpha} E_\n$, then $E$ is a null subset of $G \times G$, and for any $(a,b) \in G \times G \backslash E$, we see upon taking ultralimits that
$$
 \| \frac{1}{4} (L_{a} + L_{b} + L_{a^{-1}} + L_{b^{-1}}) f\|_{L^2(G, {\mathcal B}_G, \mu_G)} \leq (1 - \eps) \|f\|_{L^2(G, {\mathcal B}_G, \mu_G)}$$
 whenever $f$ is the standard part of a bounded internal function of mean zero, giving \eqref{loo} for $\mu_{G \times G}$-almost all $(a,b)$, as required.
 
In a similar vein, to prove (iii) we use the main result of Breuillard and Gamburd \cite{breuillard-gamburd} which shows that there exists a subset $A$ of the primes of zero relative density, such that if $p_\n$ avoids $A$, then the spectral gap \eqref{alo} holds whenever $a_\n$ and $b_\n$ generate $SL_2(F_{p_\n})$.  The classification of all proper subgroups of $SL_2(F_{p_\n})$ are classical, and it is known that all such subgroups either have size $O(1)$ or else are contained in a group containing a conjugate of the Borel subgroup
$$ B(F_{p_\n}) := \{ \begin{pmatrix} a & t \\ 0 & a^{-1} \end{pmatrix}: a \in F_{p_\n} \backslash \{0\}; t \in F_{p_\n} \}$$
with index $O(1)$.  Taking ultraproducts, we see that if $(a,b) \in G \times G$ is such that \eqref{loo} fails, then $a,b$ either lie in a finite subgroup of $G$, or a group containing a conjugate of the Borel subgroup $B(F)$ with finite index.  In either case, $a,b$ lie in a proper algebraic subgroup of $SL_2$, and the claim (iii) follows.

The claim (iv) follows very similarly from \cite[Theorem 1]{bourgain-gamburd} and is left to the reader, so we turn to (v). Let $a = \lim_{\n \to \alpha} a_\n$ be an element of $G$, so that $a_\n \in SL_2(F_{p_\n})$ for an $\alpha$-large set of $\n$.  For each $\n$, we consider the cyclic subgroup $\langle a_\n \rangle$ of $SL_2(F_{p_\n})$ generated by $a_\n$. This is an abelian subgroup of $SL_2(F_{p_\n})$, and as such can easily be verified to have cardinality $O(p_\n)$.  In particular, the index of $\langle a_\n \rangle$ in $SL_2(F_{p_\n})$ (which has order comparable to $p_\n^3$) goes to infinity as $\n \to \infty$.  As such, one can form (for an $\alpha$-large set of $\n$) a subset $E_\n$ of $SL_2(F_{p_\n})$ which is the union of right cosets of $\langle a_\n \rangle$, and whose density $\mu_{SL_2(F_{p_\n})}(E_\n)$ converges to $1/2$ as $\n \to \infty$.  Setting $E := \prod_{\n \to \alpha} E_\n$, we obtain the claim. 
\end{proof}

We do not know if all ultra quasirandom groups obey the conclusion (i) of the above proposition.  However, all ultra quasirandom groups obey (v), because one can show that the index of any subgroup $H$ in a $D$-quasirandom group $G$ is at least $D$ (otherwise the quasiregular representation
on $L^2(G/H)$ would be too small of a dimension), and one can run the argument used to prove (v) above to handle the general case.  The \emph{uniform expansion conjecture} asserts that the set $A$ in (iii) can be deleted, thus unifying (ii)-(iv), but this conjecture remains open.

Finally, we observe that the failure of mild mixing that occurs in Theorem \ref{mix}(v) also occurs for other limits of finite groups than ultra quasirandom groups.  We illustrate this with the infinite alternating group $A_\infty$, which is the direct limit of the finite alternating groups $A_n$.  (This result is not used elsewhere in the paper.)

\begin{proposition}[Failure of mild mixing]  Let $S_\infty$ be the group of all bijections of $\N$ that fix all but finitely many natural numbers; this group can be viewed as the union of the finite permutation groups $S_n$, which are the subgroup which fix all natural numbers but $\{1,\ldots,n\}$.  Let $A_\infty$ be the index two subgroup of $S_\infty$ consisting of the union of the alternating groups $A_n$.  Then there exists an ergodic action of $A_\infty$ on some probability space $(X,{\mathcal X},\mu)$, a sequence $g_1,g_2,\ldots$ of distinct elements of $A_\infty$, and a subset $E$ of $X$ of measure $\mu(E)=1/2$ such that $g_n E = E$ for all $n$.
\end{proposition}

\begin{proof} (Sketch)  Let $G := \prod_{\n \to \alpha} S_\n$.  Thanks to the nesting of the $S_n$, $G$ naturally contains an embedded copy of $\bigcup_n S_n = S_\infty$ and hence $A_\infty$.  Thus $A_\infty$ acts on the probability space $(G, {\mathcal B}_G, \mu_G)$ by left shift.  This system itself is not ergodic; for instance, if $E_\n$ denotes the set of permutations in $S_\n$ that map an odd number to $\n$, then one can verify that the ultraproduct $E := \prod_{\n \to \alpha} E_\n$ has Loeb measure $1/2$ but is invariant up to null sets by the action of $S_\infty$.  However, we can create ergodic factors of this action as follows.  Let $B_\n$ denote the set of permutations in $S_\n$ that map an odd number to $1$, and let $B := \prod_{\n \to \infty} B_\n$ be the ultraproduct.  One easily verifies that $\mu_G(B) = 1/2$.  Let $(X,{\mathcal X},\mu)$ be the factor of $(G,{\mathcal B},\mu_G)$ generated by $B$ and the $A_\infty$ action, thus $X=G$, $\mu$ is the restriction of $\mu_G$ to ${\mathcal X}$, and any set in ${\mathcal X}$ can be approximated to arbitrary accuracy in $\mu_G$ by a finite boolean combination of shifts $L_g B$ of $B$ with $g \in S_\infty$.  Any such boolean combination is a set $F$ with the property that the membership of a given permutation $\sigma = \lim_{\n \to \alpha} \sigma_\n \in G$ in $F$ depends only on the parity $\sigma^{-1}(i) \hbox{ mod } 2 = \lim_{\n \to \alpha} \sigma_\n^{-1}(i) \hbox{ mod } 2$ of preimage of a finite number of natural numbers $i$.  Because of this, we will have the Bernoulli-type mixing property $\mu_G( L_g F \cap F' ) = \mu_G(F) \mu_G(F')$ for any such boolean combinations $F,F'$, provided $g \in A_\infty$ maps a certain finite set of natural numbers to sufficiently large values.  Indeed, a simple counting argument shows that if $M$ is a fixed natural number and $\sigma_\n$ is chosen at random from $A_\n$, then the parities of the $M$ quantities of $\sigma_{\n}^{-1}(\{1\}),\ldots,\sigma_\n^{-1}(\{M\})$ behave like independent Bernoulli variables in the limit $\n \to \infty$, giving the claim. This demonstrates ergodicity of the $A_\infty$ action on $(X,{\mathcal X},\mu)$.  On the other hand, if $g$ is any permutation that fixes $1$, then $E$ is fixed by $L_g$, and so mild mixing fails.
\end{proof}

\end{document}